\documentclass[12pt,leqno,draft]{amsart}
\usepackage{amsmath,color}
\usepackage[round]{natbib}
\newcommand{\R}{{\mathbb R}}
\newcommand{\Z}{{\mathbb Z}}

\newcommand{\B}{{\mathbb B}}

\newtheorem{lemma}{Lemma}[section]

\newtheorem{theorem}[lemma]{Theorem}

\newtheorem{proposition}[lemma]{Proposition}
\newtheorem{definition}[lemma]{Definition}
\newtheorem{corollary}[lemma]{Corollary}

\newcommand{\C}{\mathbb{C}} 	 
\newcommand{\N}{\mathbb{N}} 	 
\newcommand{\Hi}{\mathbb{H}} 	 

\newcommand{\Sg}{\mathbb{S}} 	 

\newcommand{\esssup}{\operatornamewithlimits{ess\text{ }sup}}
\newcommand{\essinf}{\operatornamewithlimits{ess\text{ }inf}}

\newcommand{\Intervall}{\in (\frac{1}{2},1)}

\begin{document}
\title[Limit theorems for Hilbert space-valued linear processes]
{Limit theorems for Hilbert space-valued linear processes under long range dependence}
\author[M. D\"uker]{Marie-Christine D\"uker}
\today

\address{
Fakult\"at f\"ur Mathematik, Ruhr-Universit\"at Bo\-chum, 
44780 Bochum, Germany}
\email{Marie-Christine.Dueker@ruhr-uni-bochum.de}

\thanks{Research supported by the 
Research Training Group 2131 - {\em High-dimensional Phenomena in Probability - Fluctuations and Discontinuity}.}

\begin{abstract}
Let $(X_{k})_{k \in \Z }$ be a linear process with values in a separable Hilbert space $\Hi$ given by 
$X_{k}	=\sum_{j=0}^{\infty} (j+1)^{-N}\varepsilon_{k-j}$ for each $k \in \Z$, where $N:\Hi \to \Hi$ is a bounded, linear normal operator and $(\varepsilon_{k})_{ k \in \Z }$ is a sequence of independent, identically distributed $\Hi$-valued random variables with $E\varepsilon_{0}=0$ and $E\| \varepsilon_{0} \|^2<\infty$. We investigate the central and the functional central limit theorem for $(X_{k})_{k \in \Z }$ when the series of operator norms $\sum_{j=0}^{\infty} \|(j+1)^{-N}\|_{op}$ diverges. Furthermore we show that the limit process in case of the functional central limit theorem generates an operator self-similar process.
\end{abstract}
\maketitle

\section{Introduction}
In this paper, we study long-range dependent linear processes with values in a separable Hilbert space $\Hi$.
Given a sequence of bounded linear operators $u_{j}: \Hi \to \Hi$, $j \geq 0$, and a sequence of independent, identically distributed $\Hi$-valued random variables $(\varepsilon_{k})_{ k \in \Z }$ with $E\varepsilon_{0}=0$ and $E\| \varepsilon_{0} \|^2<\infty$, we define the linear process
	\begin{equation} \label{equality_general_process}
	X_{k}=\sum_{j=0}^{\infty} u_{j}(\varepsilon_{k-j}), \hspace{0.2cm} k \in \Z.
	\end{equation}
We investigate the asymptotic distribution of the partial sums $S_{n}=\sum_{k=1}^{n} X_{k}$ and of the partial sums process $\zeta_{n}(t)=S_{\lfloor nt \rfloor} + \{nt\} X_{ \lfloor nt \rfloor +1}$ with $t \in [0,1]$, where $\lfloor \cdot \rfloor$ denotes the floor function and $\{x\}=x-\lfloor x \rfloor$.\\
The behaviour of the linear process $(X_{k})_{k \in \Z }$ depends crucially on the convergence respectively divergence of the series $\sum_{j=0}^{\infty} \|u_{j}\|_{op}$, where $\|\cdot\|_{op}$ denotes the operator norm. If $\sum_{j=0}^{\infty} \|u_{j}\|_{op}<\infty$, the process $(X_{k})_{k \in \Z }$ is short range dependent. In this case, the central limit theorem holds with the usual normalizing sequence $n^{-\frac{1}{2}}$ and the normalized partial sums converge in distribution to an $\Hi$-valued Gaussian random element (see \citet{Rackauskas2010} and \citet{Merlevede_Sharp}). We are interested in the situation when the series diverges.\\
\citet{Rack_Suquet_Op_2011} investigate a functional central limit theorem for $(X_{k})_{k \in \Z }$ as in (\ref{equality_general_process}) with values in a Hilbert space $\Hi$ when $\sum_{j=0}^{\infty} \|u_{j}\|_{op}$ diverges with $u_{0}=I$ and $u_{j}=j^{-T}$ for $j \geq 1$, where $T\in L(\Hi)$ satisfies $\frac{1}{2}I<T<I$ and is self-adjoint. Additionally they assume that the operator $T$ commutes with the covariance operator of $\varepsilon_{0}$.\\
\citet{Chara_Rack_Operator,Chara_Rack_Central} consider $(X_{k})_{k \in \Z }$ with values in the Hilbert space $L_{2}(\mu)=L_{2}(\mathbb{S},\mathcal{S},\mu)$ of square-integrable real-valued functions, where $(\mathbb{S},\mathcal{S},\mu)$ is a $\sigma$-finite measure space. They choose $u_{j}=(j+1)^{-D}$ without requiring that the operator commutes with the covariance operator of $\varepsilon_{0}$. In their case $D$ is a multiplication operator given by $Df=\{d(s)f(s) | s\in \Sg\}$ for each $f \in L_{2}(\mu)$ for a measurable function $d:\Sg \to \R$.\\
We combine both results, constructing a process with values in a complex Hilbert space $\Hi$ with inner product $\langle \cdot, \cdot \rangle$ and the corresponding norm $\| \cdot \|$, choosing 
	\begin{equation} \label{equality_choice_uj}
	u_{j}=(j+1)^{-N}
	\end{equation}
for each $j\geq0$, where $N\in L(\Hi)$ is a normal operator, i.e. $N$ commutes with its hermitian adjoint denoted by $N^{*}$, that is $NN^{*}=N^{*}N$.\\
To be more precise we give some details about operators. Let $A \in L(\Hi)$, then it is called non-negative if $\langle Ax,x \rangle \geq 0$ for all $x \in \Hi$. For an additional operator $B \in L(\Hi)$ the inequality $A\geq B$ means $A-B \geq 0$. We set $\exp(A)=\sum_{k=0}^{\infty}\frac{A^k}{k!}$ and $a^{A}=\exp(A\log a)$ for $a>0$. For further details about operators we refer to \citet{comway1994course} and \citet{Linear_operators_Hilbert_space}.\\
Our main results establish sufficient conditions for a central and a functional central limit theorem. More precisely we show convergence in distribution of $n^{-H}S_{n}$ in $\Hi$ and of $n^{-H}\zeta_{n}$ in the space $C([0,1],\Hi)$ to a Gaussian stochastic process with $H=\frac{3}{2}I-N$, where $N$ is a normal operator and $C([0,1],\Hi)$ is the Banach space of continuous functions $x:[0,1]\to \Hi$ endowed with the norm $\|x\|=\sup_{0\leq t \leq1} \|x(t)\|$.
\vspace*{0.2cm}\\
As in \citet{Chara_Rack_Operator} we get an operator self-similar process. Such processes were first introduced by \citet{Lamperti1962} and play an important role in the context of long memory. Later operator self-similar processes were studied by \citet{Laha1981}. In our case we get a self-similar process with values in a complex Hilbert space $\Hi$. With this in mind, we repeat the definition of self-similarity of Hilbert space-valued random sequences referring to \citet{Matache_Operator}.
	\begin{definition} \ \\
	A stochastic process $\{Y(t) | t \geq 0\}$ on a Hilbert space $\Hi$ is called operator self-similar, if there exists a family $\{T(a)|a>0\}\subset L(\mathbb{H})$, such that
		\[
		\{ Y(at) | t \geq 0 \}\overset{f.d.d.}{=} \{T(a)Y(t) | t \geq 0 \},
		\]
for each $a>0$, where $\overset{f.d.d.}{=}$ denotes the equality of the finite-dimensional distributions.
	\end{definition}
The set $\{T(a)|a>0\}\subset L(\mathbb{H})$ is also called scaling family of operators. If $T(a)=a^{G}I$, where $G$ is a fixed scalar and $I$ is the identity operator, the process is called self-similar.
\vspace*{0.2cm}\\
In the following section we first present our main results with sufficient conditions for the central and the functional central limit theorem for the process $(X_{k})_{k \in \Z }$ with values in a general Hilbert space $\Hi$ constructed as in (\ref{equality_general_process}) with $(u_{j})_{j \geq 0}$ given by (\ref{equality_choice_uj}). In section \ref{section_application} we give an application to a convolution operator. Next, we
present an extension of the results given in \citet{Chara_Rack_Operator, Chara_Rack_Central}, which are needed to proof our main results. Especially we consider a process $(X_{k})_{k \in \Z }$ with values in the Hilbert space $L_{2}(\mu,\C)$ of square-integrable complex-valued functions.\\
In section \ref{section_proofs_of_the_main_results} we present the proofs of our main results, including the proof of the existence of an operator self-similar process with values in a Hilbert space $\Hi$.\\
The appendix consists of the proofs of the preliminary results given in section \ref{section_preliminary_results}.

\section{Main results} \label{section_main_results}
Before presenting our main results, we need some introducing definitions and preliminary results.\\
The spectral theorem for normal operators (see \citet[chapter 9, theorem 4.6]{comway1994course}) states that it is possible to decompose each normal operator $N\in L(\Hi)$ into a unitary operator $U:\Hi \to L_{2}(\mu,\C)$ and a multiplication operator $D:L_{2}(\mu,\C) \to L_{2}(\mu,\C)$. More precisely there exist a $\sigma$-finite measure space $(\Sg,\mathcal{S},\mu)$ and a unitary operator $U: \Hi \rightarrow L^{2}(\mu,\C)$ together with a bounded function $d: \Sg \rightarrow \C$, such that
		\begin{align} \label{equality_unitary_decomposition}
		UNU^{*}=D,
		\end{align}
where $D$ is a multiplication operator given by $Df=\{d(s)f(s)| s \in \Sg \}$ for each $f \in L^{2}(\mu,\C)$ with $d:\Sg \to \C$. We denote by
	\begin{equation}
	h(s)=\frac{1}{2}(d(s)+\overline{d(s)})
	\end{equation} 
the real part of $d(s)$ and write $d(r,s)=d(r)+\overline{d(s)}$. It is well known that the so called beta function is a function of two complex numbers $a,b$ with positive real part defined by $\operatorname{Beta}(a,b)
	=\int_{0}^{1}x^{a-1}(1-x)^{b-1}dx$. It may be also written as $\operatorname{Beta}(a,b)
	=\int_{0}^{\infty}x^{a-1}(x+1)^{-(a+b)}dx$.
We define the function $c: \Sg \times \Sg \to \C $ by
	\begin{align} \label{equality_beta}
	c(r,s)
	=\operatorname{Beta}(1-d(r),d(r,s)-1)
	=\int_{0}^{\infty}x^{-d(r)}(x+1)^{-\overline{d(s)}}dx.
	\end{align}
and introduce the further notations
\begin{align*}
	\sigma_{U}(r,s)
:=	E((U\varepsilon_{0})(r)(U\varepsilon_{0})(s)),
	\hspace*{0.2cm}
	\sigma^2_{U}(s)
:=	E|(U\varepsilon_{0})(s)|^2
	\hspace*{0.2cm}
	r,s \in \Sg.
\end{align*}
If $h(s)>\frac{1}{2}$ $\mu$-almost surely, the series
	\begin{align} 
	X_{k}	=\sum_{j=0}^{\infty} (j+1)^{-N}\varepsilon_{k-j},
	\end{align}
which defines the process $(X_{k})_{k \in \Z }$, converges. We postpone the proof to section \ref{section_proofs_of_the_main_results}.\\
Now, we are ready to present the central limit theorem.
	\begin{theorem} \label{theorem_central_normal_op}
	Suppose that $h(s) \Intervall$ for each $s\in \Sg$ and that the integrals
		\begin{align*}
		\int_{\Sg} \frac{\sigma^2_{U}(s)}{(1-h(s))^{2}}\mu(ds)
		\hspace*{0.2cm} \text{ and } \hspace*{0.2cm}
		\int_{\Sg} \frac{\sigma^2_{U}(s)}{(1-h(s))(2h(s)-1)} \mu(ds)
		\end{align*}
	are finite. Then
		\begin{align*}
		n^{-H}S_{n} \overset{D}{\longrightarrow} G \hspace*{0.2cm} \text{ as } \hspace*{0.2cm} n \to \infty,
		\end{align*}
	in $\Hi$ with $H=\frac{3}{2}I-N$. The Hilbert space-valued random variable $G$ is Gaussian with covariance operator $C_{G}:
	\Hi	\rightarrow \Hi$ defined by
		\begin{align*}
		C_{G}(x)=
			U^{*} \int_{\Sg} 		
			\left(\frac{c(r,s)+\overline{c(s,r)}}{(2-d(r,s))(3-d(r,s))} \sigma_{U}(r,s)\right)
			(Ux)(s) \mu(ds).
		\end{align*}
	\end{theorem}
Before we continue with the functional central limit theorem we need some further preliminaries.
We define the Gaussian stochastic process $\mathcal{G}=\{\mathcal{G}(t)| t \in \R_{\geq 0} \}$ with the help of its covariance operator $C_{\mathcal{G}}: \Hi \rightarrow \Hi$ given by
		\begin{align*}
		C_{\mathcal{G}}(x) = U^{*} \int_{\Sg} 		V_{U}((r,t)(s,u)) (Ux)(s) \mu(ds),
		\end{align*}
where
		\begin{align*}
		&V_{U}((r,t),(s,u))\\
		=&\frac{\sigma_{U}(r,s)}{(2-d(r,s))(3-d(r,s))}
			[ \overline{c(s,r)}t^{3-d(r,s)}+c(r,s) u^{3-d(r,s)}-C(r,s;t-u) \vert t-u \vert ^{3-d(r,s)} ]
		\end{align*}
and
\begin{align*}
	C(r,s;t)=
	\begin{cases}
		c(r,s)					& \text{, if } t<0 \vspace*{0.17cm}\\
		\overline{c(s,r)}	& \text{, if } t>0
	\end{cases}.
\end{align*}
	\begin{theorem} \label{theorem_functional_central_normal_op}
	Suppose that $h(s) \Intervall$ for each $s\in \Sg$, the integrals
		\begin{align*}
		E\left( \int_{\Sg} \frac{|(U\varepsilon_{0})(s)|^{2}}{(1-h(s))^2} \mu(ds)\right) ^{\frac{p}{2}}
		\hspace*{0.2cm} \text{ and } \hspace*{0.2cm}
		\int_{\Sg} \frac{\sigma_{U}^2(v)}{(1-h(s))(2h(s)-1)} \mu(ds)
		\end{align*}
	are finite and $p=2$ and $\bar{h}=\esssup_{s \in \Sg} h(s)<1$ or $p>2$. Then
	\begin{align*}
	n^{-H}\zeta_{n}\overset{D}{\rightarrow} G 
	\hspace*{0.2cm} \text{ as } \hspace*{0.2cm} 
	n\rightarrow \infty
	\end{align*}
in $C([0,1], \Hi)$ with $H=\frac{3}{2}I-N$. The process $G=\{G(t)| t \in [0,1]\}$ is the restriction of $\mathcal{G}$ to the unit interval.
	\end{theorem}	


\section{Application: Convolution operator} \label{section_application}
An example of a normal operator is the so called convolution operator $F: L_{2}(\R) 	\rightarrow 	L_{2}(\R)$ defined by
	\begin{align}
	F(f)= \int_{\R} K(x-y)f(y)dy=:K*f(x)
	\end{align}
with $K \in L_{1}(\R) \cap L_{2}(\R)$, where $L_{p}(\R) := \{f:\R \to \C | f \text{ measurable and } \int_{\R} |f(x)|^p dx < \infty\}$. It should be noted that the operator is self-adjoint, if the kernel function $K$ is hermitian, i.e. if $K(-x)=\overline{K(x)}$.\\
The Fourier transform defined by $(\mathcal{F}g)(s)=\frac{1}{\sqrt{2\pi}} \int_{\R} g(t) e^{-ist}dt$ with $s \in \R$ factorizes the convolution, i.e. $ \mathcal{F}(K*f)=\sqrt{2\pi}\mathcal{F}(K) \cdot \mathcal{F}(f) $. Defining the multiplication operator $D:	L_{2}(\R) \to 	L_{2}(\R), f\mapsto d \cdot f$ by $d(s)=\sqrt{2\pi} (\mathcal{F}K)(s)$, we obtain the spectral decomposition in (\ref{equality_unitary_decomposition}) of a convolution operator $F$, given by $\mathcal{F} F \mathcal{F}^{-1}=D$.\\
Choosing $u_{j}=(j+1)^{-F}$ in (\ref{equality_general_process}) the process $(X_{k})_{ k \in \Z }$ is defined by
	\begin{align*}
	X_{k}	= \sum_{j=0}^{\infty} (j+1)^{-\mathcal{F}^{-1}D\mathcal{F}}\varepsilon_{k-j}.
	\end{align*}
To formulate the central limit theorem for the above process, we maintain the notations introduced in the previous section
\begin{align*}
	\sigma_{\mathcal{F}}(r,s)
:=	E((\mathcal{F}\varepsilon_{0})(r)(\mathcal{F}\varepsilon_{0})(s)),
	\hspace*{0.2cm}
	\sigma^2_{\mathcal{F}}(s)
:=	E|(\mathcal{F}\varepsilon_{0})(s)|^2
	\hspace*{0.2cm}
	r,s \in \Sg
\end{align*}
and get the following corollary.
\begin{corollary}
Suppose that the real part of $\sqrt{2\pi}\mathcal{F}K(s)$ takes values in the interval $(\frac{1}{2},1)$ for each $s\in \Sg$ and that the integrals
		\begin{align*}
		\int_{\R} \frac{\sigma^2_{\mathcal{F}}(s)}{(1-h(s))^{2}}ds
		\hspace*{0.2cm} \text{ and } \hspace*{0.2cm}
		\int_{\R} \frac{\sigma^2_{\mathcal{F}}(s)}{(1-h(s))(2h(s)-1)} ds
		\end{align*}
	are finite. Then
		\begin{align*}
		n^{-H}S_{n} \overset{D}{\longrightarrow} G \hspace*{0.2cm} \text{ as } \hspace*{0.2cm} n \to \infty,
		\end{align*}
	in $L_{2}(\R)$ with $H=\frac{3}{2}I-F$. The $L_{2}(\R)$-valued random variable $G$ is Gaussian with covariance operator $C_{G}:
	L_{2}(\R)	\rightarrow L_{2}(\R)$ defined by
		\begin{align*}
			C_{G}(x)
		=	&
			\mathcal{F}^{-1} \int_{\R} 		
			\left(\frac{c(r,s)+\overline{c(s,r)}}{(2-d(r,s))(3-d(r,s))} \sigma_{\mathcal{F}}(r,s)\right)
			(\mathcal{F}x)(s) ds\\
		=	&
			\int_{\R} \int_{\R} 	\int_{\R}
			\frac{ \int_{0}^{\infty}x^{-\sqrt{2\pi}\mathcal{F}K(t_{1})}(x+1)^{-\sqrt{2\pi}\mathcal{F}^{-1}\overline{K(s)}}+x^{-\sqrt{2\pi}\mathcal{F}^{-1}\overline{K(s)}}(x+1)^{-\sqrt{2\pi}\mathcal{F}K(t_{1})}dx 
			}
			 {(2-\sqrt{2\pi}(\mathcal{F}K(t_{1})+\mathcal{F}^{-1}\overline{K(s)})(3-\sqrt{2\pi}(\mathcal{F}K(t_{1})+\mathcal{F}^{-1}\overline{K(s)})} 
			\\&			 
			 \sigma_{\mathcal{F}}(t_{1},s)
			 e^{-i(st_{2}-rt_{1})} x(t_{2}) dt_{2}   dt_{1} ds .
		\end{align*}
\end{corollary}
For an explicit example we take the kernel function $K(x)=e^{-a|x|}+\frac{\delta(x)}{2}$ with $a>0$, where $\delta(x)$ denotes the Dirac delta function. So the convolution operator $F: L_{2}(\R) 	\rightarrow 	L_{2}(\R)	$ is defined by
	\begin{align*}
		F(f)(x)
	=	\int_{\R} (e^{-a|x-y|}+\delta(x-y))f(y)dy
	=	\int_{\R} e^{-a|x-y|}f(y)dy+f(x).
	\end{align*}
The Fourier transform equals $\mathcal{F}K(s)=\frac{1}{\sqrt{2\pi}}(\frac{2a}{a^2+s^2}+\frac{1}{2})$ and takes values in the interval $(\frac{1}{2\sqrt{2\pi}},\frac{1}{\sqrt{2\pi}}(\frac{2}{a}+\frac{1}{2}))$. So $h(s)=\sqrt{2\pi}\mathcal{F}K(s) \Intervall$, if $a>4$, and the first assumption of the above corollary is fulfilled.

\section{Preliminary results} \label{section_preliminary_results}
\citet{Chara_Rack_Central,Chara_Rack_Operator} investigate a central and a functional central limit theorem for a linear process $(X_{k})_{k \in \Z }$ in form of (\ref{equality_general_process}) with values in the real Hilbert space $L_{2}(\mu)$ of square-integrable real-valued functions.
We extend their result to the complex Hilbert space of square-integrable complex-valued functions denoted by $L_{2}(\mu,\C)$ with inner product
	\begin{align*}
	\langle f,g \rangle = \int_{\Sg} f(s)\overline{g(s)} \mu(ds), \hspace*{0.2cm} f,g \in L_{2}(\mu,\C).
	\end{align*}
With this in mind, we choose a complex-valued multiplication operator defined by \\
$Df=\{d(s)f(s)|s \in \Sg \}$ for each $f \in L_{2}(\mu,\C)$ with $d:\Sg \to \C$ and consider the process (\ref{equality_general_process}) with 
	\begin{align} \label{equality_uj_komplex_op}
	u_{j}=(j+1)^{-D}.
	\end{align}
We denote $h(s)=\frac{1}{2}(d(s)+\overline{d(s)})$ and introduce the notations
	\begin{align*}
	\sigma(r,s):=E(\varepsilon_{0}(r)\overline{\varepsilon_{0}(s)}) , 
	\hspace*{0.2cm} 
	\sigma^2(s):=E|\varepsilon_{0}(s)|^2,
	\hspace*{0.2cm}
	r,s \in \Sg.
	\end{align*}
We start with the central limit theorem.

\begin{theorem} \label{theorem_complex_multiplication_operator_central_limit}
Suppose that $h(s) \Intervall$ for each $s\in \Sg$ and that the integrals
	\begin{align*}
	\int_{\Sg} \frac{\sigma^2(s)}{(1-h(s))^{2}}\mu(ds)
	\hspace{0.2cm} \text{ and } \hspace{0.2cm} 
	\int_{\Sg}\frac{\sigma^2(s)}{(1-h(s))(2h(s)-1)}\mu(ds)
	\end{align*}
are finite. Then
	\begin{align*}
	n^{-H}S_{n} \overset{D}{\rightarrow} G \hspace*{0.2cm} \text{ as } \hspace*{0.2cm} n \to \infty,
	\end{align*}
in $L_{2}(\mu,\C)$ with $H=\frac{3}{2}I-D$.
The $L_{2}(\mu,\C)$-valued random variable $G$ is Gaussian with zero mean and covariance
	\begin{align*}
	E(G(r)\overline{G(s)})=\frac{c(r,s)+\overline{c(s,r)}}{(3-d(r,s))(2-d(r,s))} \sigma(r,s),
	\end{align*}
with $c(r,s)=\operatorname{Beta}(1-d(r),d(r,s)-1)$, see (\ref{equality_beta}).
\end{theorem}
The proof of this theorem may be found in the appendix.\\
Before continuing with the functional central limit theorem, we introduce the function $V: \mathbb{T} \times \mathbb{T}	\rightarrow 	\C$ with $\mathbb{T}=\Sg \times \R_{\geq 0}$ and
	\begin{align} \label{equality_function_V}
	&V((r,t),(s,u))  \\
	=&\frac{\sigma(r,s)}{(3-d(r,s))(2-d(r,s))} [\overline{c(s,r)}t^{3-d(r,s)}+c(r,s) u^{3-d(r,s)}-C(r,s;t-u) \vert t-u \vert ^{3-d(r,s)}] \nonumber,
	\end{align}
where
	\begin{align*}
	C(r,s;t)=
		\begin{cases}
		c(r,s)					& \text{, if } t<0	\vspace*{0.17cm}\\
		\overline{c(s,r)}	& \text{, if } t>0
		\end{cases}.
	\end{align*}
Especially there exists a Gaussian stochastic process $\mathcal{G}=\{\mathcal{G}(s,t)| (s,t)\in \mathbb{T}\}$ with zero mean and covariance function $V$. For more details see Lemma \ref{Lemma_hermitesch_positiv_definit}.
\begin{theorem} \label{theorem_complex_multiplication_operator_functional_central_limit}
Suppose that $h(s)\in (\frac{1}{2},1)$ for each $s\in\mathbb{S}$, the integrals 
	\begin{equation*}
	E\left[ \int_{\mathbb{S}}\frac{|\varepsilon_{0}(v)|^{2}}{[1-h(v)]^2}\mu(dv)\right]^{\frac{p}{2}}
	\hspace*{0.2cm} \text{ and} \hspace*{0.2cm} 
	\int_{\mathbb{S}}\frac{\sigma^{2}(v)}{[1-h(v)][2h(v)-1]}\mu(dv)
	\end{equation*}
are finite and either $p=2$ and $\bar{h}=\esssup_{s \in \Sg} h(s)<1$ or $p>2$. Then
	\begin{align*}
	n^{-H}\zeta_{n}\overset{D}{\rightarrow} G \text{ as } n\rightarrow \infty
	\end{align*}
in $C([0,1], L_{2}(\mu,\C))$ with $H=\frac{3}{2}I-D$. The process $G=\{G(s,t)| (s,t)\in \Sg \times [0,1] \}$ is a restriction to $\Sg \times [0,1]$ of the process $\mathcal{G}$.
\end{theorem}

Again we refer to the appendix for the proof details.

\section{Proofs of the main results} \label{section_proofs_of_the_main_results}
We start with some general results about the process $(X_{k})_{ k \in \Z }$ given by (\ref{equality_general_process}) with $u_{j}=(j+1)^{-N}$. First, we need to show the convergence of the series. Therefore we rewrite the process $(X_{k})_{k \in \Z }$ with the help of (\ref{equality_unitary_decomposition}) and obtain
	\begin{align} 
	X_{k}	=\sum_{j=0}^{\infty} (j+1)^{-N}\varepsilon_{k-j}	
				=\sum_{j=0}^{\infty} (j+1)^{-U^{*}DU}\varepsilon_{k-j}	
				=\sum_{j=0}^{\infty} U^{*}\left( (j+1)^{-D}(U\varepsilon_{k-j})\right).
	\end{align}
To avoid confusion we denote the inner product of $L^{2}(\mu,\C)$-space as $\langle \cdot,\cdot \rangle_{2}$ and the corresponding norm as $\| \cdot \|_{2}$.\\
The series of operator norms 
$\sum_{j=0}^{\infty} \|(j+1)^{-N}\|_{op}$ diverges if and only if $\essinf_{s \in \Sg} h(s) \leq 1$, because
		\begin{align*}
		\|u_{j}\|_{op}
		=& \sup\left\lbrace \frac{ \|(j+1)^{-N}f \|_{\Hi}	}{ \|f\|_{\Hi} }	| f \in \Hi \text{ with } f \neq 0\right\rbrace	\\
		=& \sup\left\lbrace \frac{\|(j+1)^{-D}(Uf) \|_{2}}{\| U f\|_{2}} | Uf \in L_{2}(\mu, \C) \text{ with } Uf \neq 0\right\rbrace	
		=(j+1)^{-\essinf_{s \in \Sg}h(s)},
		\end{align*}
where the last step follows, since the operator norm of the multiplication operator $(j+1)^{-D}$ is known. Therefor we refer to the appendix, especially to (\ref{equality_operatornorm_multiplicationoperator}).
\vspace*{0.2cm}\\
We define the process $(Z_{k})_{k \in \Z }$ by 
	\begin{align} \label{equality_process_Z}
	Z_{k}=\sum_{j=0}^{\infty} (j+1)^{-D}(U\varepsilon_{k-j}).
	\end{align}
Our aim is to apply the results of Theorems \ref{theorem_complex_multiplication_operator_central_limit} and \ref{theorem_complex_multiplication_operator_functional_central_limit} to the new process $(Z_{k})_{k \in \Z }$. With this in mind, we prove that the series $(U\varepsilon_{k})_{ k\in \Z }$ fulfills the assumptions.
Since $(\varepsilon_{k})_{ k\in \Z }$ is a sequence of independent, identically distributed random variables with values in $\Hi$ and $U: \Hi \to L_{2}(\mu,\C)$ is a unitary operator, $(U\varepsilon_{k})_{ k\in \Z }$ is a sequence of $L_{2}(\mu,\C)$-valued random variables.\\
Moreover, the expected value is zero since $E\left( U\varepsilon_{0} \right) = U(E\varepsilon_{0})=0$ holds.\\
We still need to verify the interchangeability of the expected value and the unitary operator.
Referring to properties of expectation in \citet{bosq2000linear} about the interchange of an operator and the expectation of a Hilbert space-valued random variable, it suffices to prove if \\
$\varepsilon_{0} \in L^{1}_{\Hi}(P):=\{ X | \| X \|_{1}=E \| X \|<\infty \}$.
This easily follows from
	\begin{align*}
		E\|\varepsilon_{0}\|
	=		E\left( \|\varepsilon_{0}\|^{2} \right) ^{\frac{1}{2}}
	\leq		\left( E\|\varepsilon_{0}\|^{2} \right) ^{\frac{1}{2}}<\infty.
	\end{align*}
Using unitarity of $U$, we get the finite second moments
	\begin{align*}
	E\|U\varepsilon_{0}\|^2_{2}=E\|\varepsilon_{0}\|^2<\infty.
	\end{align*}
So, referring to the assumptions in Theorems \ref{theorem_central_normal_op} and \ref{theorem_functional_central_normal_op}, the process $(Z_{k})_{ k \in \Z }$ fulfils the assumptions of Theorems \ref{theorem_complex_multiplication_operator_central_limit} and \ref{theorem_complex_multiplication_operator_functional_central_limit}, and we are able to apply Lemma \ref{Lemma_convergence_in_mean_square_complex}. Therefore, the process converges in $L_{2}(\mu,\C)$ if and only if $h(s)>\frac{1}{2}$ for $\mu$-almost all $s\in \mathbb{S}$ and if the integral
	\begin{align*}
	\int_{\Sg} \frac{\sigma_{U}^2(s)}{(1-h(s))^{2}}\mu(ds).
	\end{align*}
is finite.
The almost sure convergence and boundedness of the unitary operator enables the interchangeability of the series and the hermitian adjoint of $U$, so we obtain
	\begin{align} \label{equality_process_normal_perator_after_interchange}
	X_{k}=U^{*}\left( \sum_{j=0}^{\infty} (j+1)^{-D}(U\varepsilon_{k-j})\right)
	\end{align}
and its convergence. \\
The rest of the section is divided into two parts. First we prove the central limit theorem, secondly the functional central limit theorem. Additionally we investigate a self-similar Gaussian process with values in $\Hi$.\\
Before we start let us recall some properties of random variables with values in abstract spaces.
\vspace*{0.2cm}\\
With reference to \citet{ledoux1991probability}
a random variable $Y$ with values in $\Hi$ is called Gaussian, if for any continuous linear function $f:\Hi \to \C$, $f(Y)$ is a Gaussian random variable.
A stochastic process $(Y_{t})_{ t \in T }$ with values in $\Hi$ is Gaussian, if each linear combination $\sum_{i=1}^{n}a_{i}Y_{t_{i}}$ is a Gaussian random element, i.e. for each $n\geq1$, $a_{1},...,a_{n}$ in $\C$ and $t_{1},...,t_{n}$ in $T$.

\begin{proof}[Proof of Theorem~\ref{theorem_central_normal_op}]
As announced above we want to apply Theorem \ref{theorem_complex_multiplication_operator_central_limit} to the process $(Z_{k})_{ k \in \Z }$ defined by (\ref{equality_process_Z}). Under our assumptions we get the convergence in distribution of the partial sums $\tilde{S}_{n}=\sum_{k=1}^{n} Z_{k}$ normalized by $n^{-(\frac{3}{2}I-D)}$ to a Gaussian zero mean process $\tilde{G}=\{\tilde{G}(s) | s \in \Sg\}$ with covariance function
	\begin{align}	\label{equality_covariancenfunction_tilde{G}}
		E(\tilde{G}(r)\overline{\tilde{G}(s)})
	=\frac{c(r,s)+\overline{c(s,r)}}{(2-d(r,s))(3-d(r,s))} \sigma_{U}(r,s).
	\end{align}
In other words
	\begin{align}	\label{equality_weak_convergence_prozess_partial_sums_Zn}
	n^{-(\frac{3}{2}I-D)}\tilde{S}_{n}
	\overset{D}{\to}
	\tilde{G}
	\hspace*{0.7cm}
	\text{ as }
	n \to \infty
	\text{ in }
	L_{2}(\mu,\C).
	\end{align}
Rewriting the expression $n^{-(\frac{3}{2}-N)}S_{n}$ with $S_{n}=\sum_{k=1}^{n}X_{k}$ and help of the spectral theorem, we get
	\begin{align*}
	n^{-(\frac{3}{2}I-N)}S_{n}
	=&U^{*}(n^{D-\frac{3}{2}}
	(U S_{n}))
	=U^{*}\left( n^{D-\frac{3}{2}}
	\Big( U \sum_{k=1}^{n} (U^{*}Z_{k})\Big) \right) 	\\
	=&	U^{*}\left( n^{-(\frac{3}{2}-D)} \sum_{k=1}^{n} Z_{k} \right)
	= U^{*}\left(  n^{-(\frac{3}{2}-D)}\tilde{S}_{n}\right) .
	\end{align*}
Denote $z_{n}:=n^{\frac{3}{2}-D}$. 
We use the continuous mapping theorem and the fact that $U^{*}: L_{2}(\mu,\C) \rightarrow \Hi$ is a unitary operator, i.e. it is bounded and so particularly continuous. Applying $U^{*}: L_{2}(\mu,\C) \rightarrow \Hi$ to the convergence (\ref{equality_weak_convergence_prozess_partial_sums_Zn}), we get
	\begin{align*}
	U^{*}\left( n^{-(\frac{3}{2}I-D)}\tilde{S}_{n} \right) 
	\overset{D}{\to}
	U^{*}\tilde{G}
	\hspace*{0.7cm}
	\text{ as }
	n \to \infty
	\text{ in }
	\Hi.
	\end{align*}
So the limit process is of the form $G=\{U^{*}\tilde{G} \}$ with values in $\Hi$. Referring to Theorem \ref{theorem_complex_multiplication_operator_central_limit} $\{\tilde{G}(s)|s \in \Sg \}$ is Gaussian with values in $L_{2}(\mu,\C)$, i.e. $\tilde{f}(\tilde{G})$ is a complex-valued Gaussian random element for each continuous, linear function $\tilde{f}:L_{2}(\mu,\C) \to \C$. So
	\begin{align*}
	f(G)=f(U^{*}\tilde{G})=(f \circ U^{*})(\tilde{G})
	\end{align*}
is Gaussian for each continuous, linear function $f:\Hi \to \C$, since $f \circ U^{*}:L_{2}(\mu,\C) \to \C$.
\vspace*{0.2cm}\\
In the final step of the proof we calculate the covariance operator of the limit process. In general the covariance operator $C_{G}:\Hi \rightarrow \Hi$ is given by $C_{G}(x)=E(\langle G,x \rangle \overline{G})$ (see \citet{bosq2000linear}). An alternative definition says that $C_{G}$ is the covariance operator of $G$ if and only if $E(\langle G,x \rangle \overline{\langle y,G\rangle})
= \langle x, C_{G}y\rangle$ for each $x,y \in \Hi$.
Therefor we use equality (\ref{equality_covariancenfunction_tilde{G}}) and obtain
	\begin{align*}
		E(\langle G,x \rangle \overline{\langle y,G\rangle})
	=&	E(\langle Ux, \tilde{G} \rangle_{2} \langle \tilde{G},Uy \rangle_{2})
	=	E\left( \int_{\Sg} (Ux)(r) \overline{\tilde{G}(r)} \mu(dr)\int_{\Sg} \tilde{G}(s)\overline{(Uy)(s)} \mu(ds)\right)	\\
	=	&\int_{\Sg}\int_{\Sg}  (Ux)(r) \overline{(Uy)(s)}
		\overline{E(\tilde{G}(r)\overline{\tilde{G}(s)})}
		 \mu(dr)\mu(ds)	\\
	=	&\langle Ux, \int_{\Sg}
		E(\tilde{G}(\cdot)\overline{\tilde{G}(s)})
		(Uy)(s)\mu(ds) \rangle_{2}	\\
	=	&\langle x, U^{*} \int_{\Sg}
		\left(\frac{c(r,s)+\overline{c(s,r)}}{(2-d(r,s))(3-d(r,s))} \sigma_{U}(r,s)\right)
		(Uy)(s)\mu(ds) \rangle.
	\end{align*}
The assumed interchangeability of the expected value and the integrals easily follows from Fubini. To prove this, 
we first apply Lemma \ref{lemma_dominated convergence_integrable function_central_limt} to $n^{-(\frac{3}{2}-D)}\tilde{S}_{n}$, then the inequality also holds for the limit process $\tilde{G}$, i.e.
		\begin{align}
				E|\tilde{G}(s)|^2
		\leq		\left( \frac{\sigma_{U}^2(s)}{(1-h(s))^2} + \frac{\sigma_{U}^2(s)}{(1-h(s))(2h(s)-1)}\right)=:\tilde{g}(s).
		\end{align}
Using H\"older's inequality yields
	\begin{align*}
			&\int_{\Sg} \int_{\Sg} E|(Ux)(r)\overline{(Uy)(s)}\overline{\tilde{G}(r)}\tilde{G}(s)|\mu(dr)\mu(ds)	\\
	\leq	&\int_{\Sg} \int_{\Sg} |(Ux)(r)\overline{(Uy)(s)}|
			\tilde{g}(r)^{\frac{1}{2}} \tilde{g}(s)^{\frac{1}{2}}\mu(dr)\mu(ds)	\\
	\leq	&	\left( \int_{\Sg} |(Ux)(r)|^2\mu(dr)\right) ^{\frac{1}{2}} 
				\left(\int_{\Sg} |\tilde{g}(r)| \mu(dr)\right) ^{\frac{1}{2}}
				\left( \int_{\Sg} |(Uy)(s)|^2\mu(ds)\right) ^{\frac{1}{2}} 
				\left(\int_{\Sg} |\tilde{g}(s)| \mu(ds)\right) ^{\frac{1}{2}}	\\
	\leq	&	\|x\| \|y\|
				\int_{\Sg} |\tilde{g}(s)| \mu(ds)	
	<			\infty.
	\end{align*}
Under our assumptions in Theorem \ref{theorem_complex_multiplication_operator_central_limit} the function $\tilde{g}$ is integrable and it follows the assertion.
\end{proof}

\begin{proof}[Proof of Theorem~\ref{theorem_functional_central_normal_op}]
First we rewrite the piecewise linear function using (\ref{equality_process_normal_perator_after_interchange})
\begin{align} \label{equality_rewritten_piecewise_linear_process}
S_{\lfloor nt \rfloor} + \{nt\} X_{\lfloor nt \rfloor +1} 
=&\sum_{k=1}^{\lfloor nt \rfloor}U^{*}Z_{k} + \{nt\} U^{*}Z_{\lfloor nt \rfloor +1}		\\
=&U^{*}\left( \sum_{j=-\infty}^{\lfloor nt \rfloor +1} a_{nj}(t) (U\varepsilon_{j})\right)
:=U^{*} \tilde{\zeta}_{n}(t)	\nonumber
\end{align}
with
$
a_{nj}(t)=\sum_{k=1}^{\lfloor nt \rfloor} v_{k-j} + \{nt\} v_{\lfloor nt \rfloor +1-j}
$
and
	\begin{align*}	v_{j}=
		\begin{cases}
		(j+1)^{-D}	&\text{, if } j \geq 0	\\
		0					&\text{, if } j < 0
		\end{cases}.
	\end{align*}

We consider the sequence of piecewise linear functions $\zeta_{n}$ and the stochastic Process $G$ as random elements in the separable Banach space $C([0,1], \Hi)$. To prove the theorem, we have to show the convergence of the finite-dimensional distributions and tightness.
\vspace*{0.2cm}\\
Analogously to the proof of the central limit theorem we first use the fact that the process $(Z_{k})_{ k \in \Z }$ fulfills the assumptions of Theorem \ref{theorem_complex_multiplication_operator_functional_central_limit}. So the process $(\tilde{\zeta}_{n}(t))_{n \in \N }$ normalized by $z_{n}=n^{\frac{3}{2}-D}$ converges to a Gaussian process $\tilde{G}=\{ \tilde{G}(s,t) | (s,t) \in \Sg \times [0,1]\}$ with zero mean and covariance function $V_{U}$ in $C([0,1], L_{2}(\mu,\C))$. As stated in section \ref{section_preliminary_results} the process $\tilde{G}$ is defined as a restriction of a Gaussian stochastic process $\tilde{\mathcal{G}}=\{ \tilde{\mathcal{G}}(s,t) | (s,t) \in \Sg \times [0,\infty)\}$.

\subsubsection{Convergence of the finite-dimensional distributions}
The statements we have shown for the proof of Theorem \ref{theorem_complex_multiplication_operator_functional_central_limit} are applicable to the process $(\tilde{\zeta}_{n}(t))_{n \in \N}$ defined in (\ref{equality_rewritten_piecewise_linear_process}) since $(Z_{k})_{ k \in \Z }$ fulfills the conditions in Theorem \ref{theorem_complex_multiplication_operator_functional_central_limit}.
So using the convergence (\ref{equality_convergence_finite_dim_functional_central}) yields
		\begin{align*}
		z_{n}^{-1}\tilde{\zeta}_{n}^{(q)}:=(z_{n}^{-1}\tilde{\zeta}_{n}(t_{1}),...,z_{n}^{-1}\tilde{\zeta}_{n}(t_{q}))
		\overset{D}{\longrightarrow}
		(\tilde{G}(t_{1}),...,\tilde{G}(t_{q}))=:\tilde{G}^{(q)}
		\hspace*{0.2cm} \text{ as } \hspace*{0.2cm} 
		n \to \infty
		\end{align*}
	in $L_{2}^{q}(\mu,\C)$ for each $q \in \N$ and $t_{1},...,t_{q} \in [0,1]$. In other words
		\begin{align*}
		z_{n}^{-1}\tilde{\zeta}_{n}^{(q)}
		\overset{D}{\longrightarrow}
		\tilde{G}^{(q)}
		\hspace*{0.2cm} \text{ as } \hspace*{0.2cm} 
		n \to \infty.
		\end{align*}
To show the convergence of the finite-dimensional distributions we have to prove
		\begin{align*}
		n^{-H}\zeta_{n}^{(q)}:=
		(n^{-H}\zeta_{n}(t_{1}),...,n^{-H}\zeta_{n}(t_{q})) 
		\overset{D}{\longrightarrow}
		(G(t_{1}),...,G(t_{q}))=:G^{(q)}
		\hspace*{0.2cm} \text{ as } \hspace*{0.2cm} 
		n \to \infty
		\end{align*}
	in $\Hi^{q}$ for each $q \in \N$ and $t_{1},...,t_{q} \in [0,1]$. Using the known notations and (\ref{equality_rewritten_piecewise_linear_process}) this is equivalent to
		\begin{align*}
		(U^{*}(z_{n}^{-1}\tilde{\zeta}_{n}(t_{1})),...,U^{*}(z_{n}^{-1}\tilde{\zeta}_{n}(t_{q})))
		\overset{D}{\longrightarrow}
		(G(t_{1}),...,G(t_{q}))
		\hspace*{0.2cm} \text{ as } \hspace*{0.2cm} 
		n \to \infty
		\end{align*}
	in $\Hi^{q}$, since  $n^{-H}=U^{*}z_{n}^{-1}U$.
	Defining the mapping
		\begin{align*} \hat{U}^{*}:
		\begin{cases}
		L_{2}^{q}(\mu,\C)		&\rightarrow	\Hi ^{q}\\
		(g_{1},...,g_{q})	&\mapsto		(U^{*}g_{1},...,U^{*}g_{q})
		\end{cases},
		\end{align*}
		it is possible to rewrite the convergence condition to
		\begin{align*}
		\hat{U}^{*}(
		(z_{n}^{-1}\tilde{\zeta}_{n}(t_{1}),...,z_{n}^{-1}\tilde{\zeta}_{n}(t_{q}))
		\overset{D}{\longrightarrow} 
		(G(t_{1}),...,G(t_{q}))
		\hspace*{0.2cm} \text{ as } \hspace*{0.2cm} 
		n \to \infty,
		\end{align*}
	in other therms
		\begin{align*}
		\hat{U}^{*}\tilde{\zeta}_{n}^{(q)}
		\overset{D}{\longrightarrow}
		G^{(q)}
		\hspace*{0.2cm} \text{ as } \hspace*{0.2cm} 
		n \to \infty.
		\end{align*}
Analogously to the proof of the central limit theorem, we get the convergence statement, using the continuous mapping theorem.
		\begin{align*}
		\hat{U}^{*}\zeta_{n}^{(q)}
		\overset{D}{\to}
		\hat{U}^{*}\tilde{G}^{(q)}
		\hspace*{0.2cm} \text{ as } \hspace*{0.2cm} 
		n \to \infty.
		\end{align*}
So, the limit process $G$ is of the form $U^{*}\tilde{G}$. 
	
\subsubsection{Tightness}
In 1968, Billingsley establishes sufficient conditions for tightness for a sequence of random elements with values in $C([0,1],\R)$. Referring to \citet{Rack_Suquet_Op_2011}, we use the following extension to the space $C([0,1],\Hi)$.
\begin{proposition} \label{proposition_tightness}
Let $\Hi$ be a separable Hilbert space.
A sequence of random elements $(Y_{n})_{n\geq1}$ with values in $C([0,1],\Hi)$ is tight if 
	\begin{itemize}
	\item[(i)] for every $t\in[0,1]$, $(Z_{n}(t))_{n\geq1}$ is tight in $\Hi$.
	\item[(ii)] there exist constants $\gamma \geq 0$, $\alpha>1$ and a continuous increasing function
					$F:[0,1] \rightarrow \R$, such that
					\begin{align*}
					P(\| Y_{n}(t)-Y_{n}(u) \| > \lambda) \leq \lambda^{-\gamma}|F(t)-F(u)|^{\alpha}.
					\end{align*}
	\end{itemize}
\end{proposition}
The first point easily follows since the central limit theorem is still proved. In detail, since the process $n^{-H}S_{n}$ converges in distribution in $\Hi$, the sequence $\{n^{-H}\zeta_{n}(t)\}$ converges in distribution in $\Hi$ and using Prohorov, it is tight on $\Hi$ for each $t\in [0,1]$
\vspace*{0.2cm} \\
The second point is an implication of Lemma \ref{lemma_tightness_complex}. Using additionally the linearity of $U^{*}$, we obtain
		\begin{align*}
			 E\| n^{-H}\zeta_{n}(t)-n^{-H}\zeta_{n}(u) \|^{p}
		=	&E\| U^{*}(z_{n}^{-1}\tilde{\zeta}_{n}(t))-
				 U^{*}(z_{n}^{-1}\tilde{\zeta}_{n}(u)) \|^{p}\\
		=	&E\| U^{*}(	z_{n}^{-1}\tilde{\zeta}_{n}(t)-
									z_{n}^{-1}\tilde{\zeta}_{n}(u)) \|^{p}	\\
		=	&E\| z_{n}^{-1}\tilde{\zeta}_{n}(t)-z_{n}^{-1}\tilde{\zeta}_{n}(u) \|^{p}_{2}	\\
		\leq &C|t-u|^{\frac{(3-2\bar{h})p}{2}},
		\end{align*}
where the last inequality follows with the help of the mentioned lemma and under the assumptions of Theorem \ref{theorem_complex_multiplication_operator_functional_central_limit}.
\end{proof}

\subsubsection*{Properties of the process $\mathcal{G}$}
Finally we show some properties of the process $\mathcal{G}=\{\mathcal{G}(t)| t \in \R_{\geq 0} \}$, especially that it is Gaussian, calculate the cross-covariances and show self-similarity.
\vspace*{0.2cm}\\
The process is Gaussian, if 
	\begin{align*}
	f\left( \sum_{i=1}^{n} a_{i} \mathcal{G}(t_{i}) \right)
	\end{align*}
is a Gaussian random element in $\C$ for each $n\geq1$, $a_{1},...,a_{n}$ in $\C$ and $t_{1},...,t_{n}$ in $\R_{\geq 0}$ and also for all continuous, linear functions $f:\Hi \rightarrow \C$.\\
Using the linearity of the unitary operator $U^{*}$, we get
	\begin{align*}
	f\left( \sum_{i=1}^{n} a_{i} \mathcal{G}(t_{i}) \right)
	=f\left( \sum_{i=1}^{n} a_{i} U^{*} \tilde{\mathcal{G}}(t_{i}) \right)
	=(f \circ U^{*}) \left( \sum_{i=1}^{n} a_{i} \tilde{\mathcal{G}}(t_{i}) \right).
	\end{align*}
The composition $f \circ U^{*}$ is a mapping from $L_{2}(\mu,\C)$ into the complex numbers. Since the process $\tilde{\mathcal{G}}=\{\tilde{\mathcal{G}}(s,t)| (s,t) \in \Sg \times \R_{\geq 0} \}$ is Gaussian, we get the assertion.
\vspace*{0.2cm}\\
The following calculation yields the cross-covariance of the process $\mathcal{G}$.
	\begin{align*}
		E(\langle x,\mathcal{G}\rangle \overline{\langle y,\mathcal{G}\rangle})
	=	&E(\langle Ux, \tilde{\mathcal{G}} \rangle_{2} \langle \tilde{\mathcal{G}},Uy \rangle_{2})	\\
	=	&\int_{\Sg} \int_{\Sg} (Ux)(r) \overline{(Uy)(s)} \overline{E(\tilde{\mathcal{G}}(r,t) \overline{\tilde{\mathcal{G}}(s,u)})} \mu(dr)\mu(ds))\\
	=	&\langle Ux, \int_{\Sg}
		V_{U}((\cdot,t),(s,u))
		(Uy)(s)\mu(ds) \rangle_{2}	\\
	=	&\langle x, U^{*} \int_{\Sg}
		V_{U}((r,t),(s,u))
		(Uy)(s)\mu(ds) \rangle
	\end{align*}
It remains to show the interchangeability of the expected value and the integrals. Applying inequality (\ref{inequality_bound_function_functional_central_limit}) to the process $\tilde{\mathcal{G}}$, we obtain
\begin{align} \label{inequality_covariances_limit_process_invariance_principle}
				&E\left( \int_{\Sg} |\tilde{\mathcal{G}}(r,t)|^2 \mu(dr)\right) \\
		\leq	&\max\{t,t^2\} \left( \int_{\Sg} \frac{\sigma_{U}^2(r)}{(1-h(r))^2}\mu(dr) + 
				\int_{\Sg} \frac{\sigma_{U}^2(r)}{(1-h(r))(2h(r)-1)} \mu(dr)\right). \nonumber
		\end{align}
With the help of Fubini's theorem, the unitarity of $U$ and H\"older's inequality we get
	\begin{align*}
			&\int_{\Sg} \int_{\Sg} E|(Ux)(r)\overline{(Uy)(s)}\overline{\tilde{\mathcal{G}}(r,t)}\tilde{\mathcal{G}}(s,u)|\mu(dr)\mu(ds)	\\
	\leq	&\int_{\Sg} \int_{\Sg} |(Ux)(r)\overline{(Uy)(s)}|
			(E|\tilde{\mathcal{G}}(r,t)|^2)^{\frac{1}{2}} (E|\tilde{\mathcal{G}}(s,u)|^2)^{\frac{1}{2}}\mu(dr)\mu(ds)\\
	\leq		&	\|x\| \|y\|
				\left(\int_{\Sg} |E|\tilde{\mathcal{G}}(r,t)|^2| \mu(ds)\right) ^{\frac{1}{2}}
				\left(\int_{\Sg} |E|\tilde{\mathcal{G}}(s,u)|^2| \mu(ds)\right) ^{\frac{1}{2}}
	<			\infty.
	\end{align*}
The finiteness follows by using inequality (\ref{inequality_covariances_limit_process_invariance_principle}) and the assumptions in theorem \ref{theorem_functional_central_normal_op}.
\vspace*{0.2cm}\\
As a last step we show the existence of an operator self-similar process with values in $\Hi$.
\begin{lemma}
The stochastic process $\{\mathcal{G}(t) | t \in [0,\infty)\}$ is operator self-similar with scaling family $\{a^{H}|a>0\}$, where $H$ is equal to $\frac{3}{2}I-N$ and $N$ is a normal operator. 
\end{lemma}
	\begin{proof}
	It is necessary to prove 		
		\[
		\{\mathcal{G}(at)| t\in [0,\infty)\}\overset{f.d.d.}{=}a^{H}\{\mathcal{G}(t)| t\in [0,\infty)\}.
		\]
	Since both sides are Gaussian processes with zero mean it suffices to show that the covariance operators are equal.
		\begin{align*}
		\langle E[\langle \mathcal{G}(at),f \rangle \overline{\mathcal{G}(au)}],g \rangle
		=\langle E[\langle a^{H} \mathcal{G}(t),f \rangle \overline{a^{H} \mathcal{G}(u)}],g \rangle
		\hspace*{0.2cm} \text{ for each } f,g \in \Hi
		\end{align*}
Using the decomposition into a unitary operator $U^{*}$ and the process $\tilde{\mathcal{G}}$ of which we still know the operator self-similarity, we get
		\begin{align*}
			\langle E[ \langle a^H \mathcal{G}(t),f \rangle \overline{a^H\mathcal{G}(u)}],g\rangle
		=	&\langle E[ \langle a^H (U^{*}\tilde{\mathcal{G}}(x,t)),f \rangle
			\overline{a^H(U^{*}\tilde{\mathcal{G}}(y,u))}],g\rangle \\
		=	&\langle E[ \langle U^{*}(a^{\frac{3}{2}-d(x)} \tilde{\mathcal{G}}(x,t)),f \rangle
			\overline{U^{*}(a^{\frac{3}{2}-d(y)}\tilde{\mathcal{G}}(y,u))}],g\rangle\\
		=	&\langle E[ \langle a^{\frac{3}{2}-d(\cdot)} \tilde{\mathcal{G}}(\cdot,t),Uf \rangle_{2} 
			\overline{U^{*}}(\overline{a^{\frac{3}{2}-d(y)}\tilde{\mathcal{G}}(y,u)})],g\rangle.
	\intertext{At this point we interchange the operator $U^{*}$ and the expected value. The regularity will checked out later.}
		=	&\langle \overline{U^{*}}(E[ \langle a^{\frac{3}{2}-d(\cdot)} \tilde{\mathcal{G}}(\cdot,t),Uf \rangle_{2} 
			\overline{a^{\frac{3}{2}-d(\cdot)}\tilde{\mathcal{G}}(\cdot,u)}]),g\rangle\\
		=	&\langle E[ \langle a^{\frac{3}{2}-d(\cdot)} \tilde{\mathcal{G}}(\cdot,t),Uf \rangle_{2} 
			\overline{a^{\frac{3}{2}-d(s)}\tilde{\mathcal{G}}(\cdot,u)}],\overline{U}g\rangle_{2} \\
		=	&\langle E[ \langle a^{\frac{3}{2}-d(\cdot)} \tilde{\mathcal{G}}(\cdot,t),\tilde{f} \rangle 
			\overline{a^{\frac{3}{2}-d(s)}\tilde{\mathcal{G}}(\cdot,u)}], \tilde{g} \rangle
	\intertext{Since $\tilde{f}:=Uf$ and $\tilde{g}:=\overline{U}g$ in $L_{2}(\mu,\C)$, the operator self-similarity of $\tilde{\mathcal{G}}$ (see Lemma \ref{lemma_operator_self_similar_multi}) is applicable and therefore}
		=	& \langle E[ \langle \tilde{\mathcal{G}}( \cdot ,at),\tilde{f} \rangle \overline{\tilde{\mathcal{G}}(s,au)}], \tilde{g} \rangle.
		\end{align*}
	As announced we prove the interchangeability of the operator $U^{*}$ and the expected value. Define $b(t):=\langle a^{\frac{3}{2}-d(\cdot)} \tilde{\mathcal{G}}(\cdot,t),Uf \rangle_{2} $. We have to prove $b(t)\overline{a^{-(\frac{3}{2}-d(y))}\tilde{\mathcal{G}}(y,u)} \in L^{1}_{L_{2}(\mu,\C)}(P)$.
		\begin{align*}
		& E \|b(t)\overline{a^{-(\frac{3}{2}-d(\cdot))}\tilde{\mathcal{G}}(\cdot,u)} \|_{2}
		=		E \left( \int_{\Sg} |b(t)\overline{a^{\frac{3}{2}-d(s)}\tilde{\mathcal{G}}(s,u)}|^2 \mu(ds)\right) ^{\frac{1}{2}}	\\
		\leq&	\left( E|b(t)|^2\right) ^{\frac{1}{2}}\left( E \int_{\Sg} |a^{\frac{3}{2}-d(s)}\tilde{\mathcal{G}}(s,u)|^2 \mu(ds)\right) ^{\frac{1}{2}}	\\
		\leq&	\left( E \int_{\Sg} |a^{\frac{3}{2}-d(r)} \tilde{\mathcal{G}}(r,t)|^2\mu(dr) \int_{\Sg} |(Uf)(r)|^2\mu(dr) \right) ^{\frac{1}{2}}
					\left( E \int_{\Sg} |a^{\frac{3}{2}-d(s)}\tilde{\mathcal{G}}(s,u)|^2 \mu(ds)\right) ^{\frac{1}{2}}	\\
		=		&	\left( \int_{\Sg} |(Uf)(r)|^2\mu(dr) \right) ^{\frac{1}{2}}
					\left( E \int_{\Sg} |a^{\frac{3}{2}-d(r)}\tilde{\mathcal{G}}(r,t)|^2 \mu(dr) \right)^{\frac{1}{2}}
					\left( E \int_{\Sg} |a^{\frac{3}{2}-d(s)}\tilde{\mathcal{G}}(s,u)|^2 \mu(ds) \right)^{\frac{1}{2}}  \\
		\leq&		\|f\|
					\max\{a^{\frac{1}{2}},a\}
					\left( \int_{\Sg} E|\tilde{\mathcal{G}}(r,t)|^2 \mu(dr) \right)^{\frac{1}{2}}
					\left( \int_{\Sg} E|\tilde{\mathcal{G}}(s,u)|^2 \mu(ds)\right) ^{\frac{1}{2}}
		<		 \infty
		\end{align*}
	We used H\"older's inequality and (\ref{inequality_covariances_limit_process_invariance_principle}).
	\end{proof}

\appendix
\section{Proofs of the preliminary results}

The proofs of theorems \ref{theorem_complex_multiplication_operator_central_limit} and \ref{theorem_complex_multiplication_operator_functional_central_limit} are closely related to the proofs in \citet{Chara_Rack_Central,Chara_Rack_Operator}. So we will focus on the passages which differ.
\vspace*{0.2cm}\\
We start with some preliminaries. It is well known that
the Beta function can be expressed as
	\begin{align*}
	 \operatorname{Beta}(x,y)=\frac{\Gamma(x)\Gamma(y)}{\Gamma(x+y)}.
	\end{align*}
Using this identity and $d(r,s)=d(r)+\overline{d(s)}$, we obtain
	\begin{align} \label{equality_Beta_function_alternative}
	\operatorname{Beta}(1-d(r),d(r,s)-1)=\frac{\Gamma(d(r,s)-1)\operatorname{Beta}(d(r),1-d(r))}{\Gamma(d(r))\Gamma(\overline{d(s)})}.
	\end{align}
For simplicity we denote $h(r,s)=h(r)+h(s)$. Recall that we write $h(s)$ for the real part of $d(s)$. Now, let $\tilde{h}(s)$ denote the imaginary part of $d(s)$. We define $c(s):=c(s,s)$, for the definition of $c(s,s)$ see (\ref{equality_Beta_function_alternative}). Using $a^{-i\tilde{h}(s)}=\exp(-i\tilde{h}(s)\log(a))$ and Euler's formula yields
	\begin{align} \label{inequality_realpart_c(s)}
	\frac{1}{2}(c(s)+\overline{c(s)})
	=& \int_{0}^{\infty}(x(x+1))^{-h(s)}\cos\Big( \log\Big( \frac{x}{x+1}\Big) ^{-\tilde{h}(s)}\Big) dx	\\
	\leq & \int_{0}^{\infty}(x(x+1))^{-h(s)}dx=:c_{h}(s).\nonumber 
	\end{align}
Continuing estimation of the right hand side gives
	\begin{align} \label{inequality_c(s)_in_dependence_of_h}
	c_{h}(s)\leq\frac{1}{1-h(s)}+\frac{1}{2h(s)-1}.
	\end{align}
Since $(u_{j})_{ j \in \Z }$ given by (\ref{equality_uj_komplex_op}) are multiplication operators in $L(L_{2}(\mu.\C))$, we have the operator norm 
	\begin{align}	\label{equality_operatornorm_multiplicationoperator}
	\|(j+1)^{-D}\|_{op}=(j+1)^{-\essinf_{s \in \Sg} h(s)}
	\end{align}
referring to \citet{comway1994course}.
So the series of operator norms $\sum_{j=0}^{\infty} \|(j+1)^{-D}\|_{op}$ diverges if and only if $\essinf_{s \in \Sg} h(s)\leq 1$. \

\begin{lemma} \label{Lemma_convergence_in_mean_square_complex}
The series (\ref{equality_general_process}) with $u_{j}$ as in (\ref{equality_uj_komplex_op}) converges in mean square if and only if $h(s)>\frac{1}{2}$ for $\mu$-almost all $s\in \mathbb{S}$ and if the integral
	\begin{align*}
	\int_{\Sg} \frac{\sigma^2(s)}{2h(s)-1}\mu(ds)
	\end{align*}
is finite. Then the series converges also almost surely.
\end{lemma}
	\begin{proof}
	We want to apply Cauchy`s criterion. Defining 
	$ X_{nk}=\sum_{j=0}^{n}(j+1)^{-D}\varepsilon_{k-j} $, we have to prove
		\begin{align*}
		\lim_{M,N \to \infty} E \| X_{Nk}-X_{Mk} \|^2=0.
		\end{align*}
	Since $(j+1)^{-D}f=\{(j+1)^{-d(s)}f(s): s\in \Sg\}$ for each $j \geq 1$ and $f \in L_{2}(\mu,\C)$
		\begin{align*}
			E \| X_{Nk}-X_{Mk} \|^2
		=	\sum_{j=M+1}^{N}\int_{\Sg} (j+1)^{-2h(s)}\sigma^2(s)\mu(ds).
		\end{align*}
	Since
		\begin{align*}
		&\frac{1}{2h(s)-1} \leq \sum_{j=1}^{\infty}j^{-2h(s)} \leq 1+\frac{1}{2h(s)-1}	
		\intertext{we get}
		&\int_{\Sg} \frac{\sigma^2(s)}{2h(s)-1} \mu(ds) 
		\leq \int_{\Sg} \sum_{j=0}^{\infty}(j+1)^{-2h(s)}\sigma^2(s)\mu(ds) \leq 
		E\|\varepsilon_{0}\|^2 + \int_{\Sg} \frac{\sigma^2(s)}{2h(s)-1} \mu(ds).
		\end{align*}
Mean square convergence implies convergence in probability and applying the L\`evy-It\^o-Nisio theorem (\citet{ledoux1991probability}) almost sure convergence follows.
	\end{proof}
We need a generalization of the geometric sum
	\begin{align}	\label{equality_generalized_geometrical_series}
	&\sum_{j=m_{1}}^{m_{2}} (m_{3}-j)e^{-jx}	  \\
	=&
	\frac{(m_{3}-a_{1})e^{(2-m_{1})x}-(m_{3}-m_{2}-1)e^{(1-m_{2})x}-(m_{3}-m_{1}+1)e^{(1-m_{1})x}+(m_{3}-m_{2})e^{-m_{2}x}}
	{(e^x-1)^2} \nonumber
	\end{align}
with $m_{i}\in \N$ for each $i \in \{1,2,3\}$ (see \citet{Rack_Suquet_Op_2011}).\\
We calculate the autocovariance function of $(X_{k}(r))_{k \geq 1}$ and $(X_{k}(s))_{k \geq 1}$, which are stationary for fixed $r$, $s$. Using $(j+1)^{-D}f=\{(j+1)^{-d(s)}f(s) | s \in \Sg \}$ we get
	\begin{align} \label{equality_autocovariance}
	\gamma_{h}(r,s)
	:=E(X_{0}(r)\overline{X_{h}(s)})	
	=&\sigma(r,s)\sum_{j=0}^{\infty}(j+1)^{-d(r)}(j+1+h)^{-\overline{d(s)}}.
	\end{align}

Before we present the proofs we cite some helpful theorems. First we refer to \citet{Cremers_Kadelka_Weak_Con}) for a theorem, which gives sufficient conditions for the weak convergence of random sequences with paths in $L_{p}(\mu,\B)$.
\begin{theorem} H. Cremers und D. Kadelka	\label{theorem_weak convergence_Cremers_Kadelka}
Let $(Y_{n})_{ n \in \N }$ be a sequence of stochastic processes with paths in $\mathcal{L}_{p}(\mu,\B)$. Then $Y_{n}\overset{D}{\rightarrow}Y$ as $n \to \infty$, if
	\begin{itemize}
	\item[(I)]
	the finite-dimensional distributions of $Y_{n}$ converge weakly to those of $Y$ almost everywhere, i.e.
		\begin{align*}
		(Y_{n}(s_{1}),...,Y_{n}(s_{k})) \overset{D}{\rightarrow} (Y(s_{1}),...,Y(s_{k}))
		\end{align*}
	for each $k\in \N$ and almost all $s_{1},...,s_{k}$ in $\Sg$,
	\item[(II)]
	$ \lim_{n \to \infty} E \|Y_{n}(s)\|^p  = E\|Y(s)\|^p $ for each $s \in \Sg$ and $n \in \N$,
	\item[(III)]
	there exists a square $\mu$-integrable function $f: \Sg \to \R$ such that $ E \|Y_{n}(s)\|^p \leq f(s)$ for each $ s\in \Sg$.
	\end{itemize}
\end{theorem}
Next we refer to \citet{Rack_Suquet_Op_2011} for a theorem, which gives sufficient conditions for the weak convergence of linear processes with values in a Hilbert space. Let $\Hi$ and $\mathbb{E}$ be two Hilbert spaces and $(\varepsilon_{j})_{ j \in \Z }$ a sequence of independent, identically distributed random variables with values in $\mathbb{E}$.
Define $(X_{n})_{ n \in \Z }$ with
		\begin{align}	\label{Prozessdarstellung_für_Beweis}
		X_n=\sum_{j\in \Z} A_{nj} \varepsilon_{j}
		\end{align}
and $A_{nj} \in L(\Hi,\mathbb{E})$. Now, we define a second process $(Y_{n})_{ n \in \Z }$ with
		\begin{align*}
		Y_n=\sum_{j\in \Z} A_{nj} \tilde{\varepsilon}_{j},
		\end{align*}
where $A_{nj}$ is the same operator as above and $\tilde{\varepsilon}_{j}$ is a sequence of Gaussian random elements with values in $\mathbb{E}$, zero-mean and the same covariance operator as $\varepsilon_{j}$.
\vspace*{0.2cm}\\
Before we state Ra\v{c}kauskas and Suquet`s lemma we need a definition of a metric on the space of probability measures on Hilbert spaces.
	\begin{definition} 
	Let $X,Y$ be $\Hi$-valued random variables, then the metric $\varrho_{k}$ is defined by 
		\begin{align*}
		\varrho_{k} (X,Y)=\sup_{f \in F_{k}} \left| Ef(X) - Ef(Y) \right|,
		\end{align*}
	where $F_{k}$ is the set of all $k$ times Fr\`{e}chet differentiable functions $f:\Hi\rightarrow \R$ such that\\
	$\sup_{x \in \Hi}|f^{(i)}(x)|\leq 1$ for $i=0,...,k$.
	\end{definition}
	No we are able to present the lemma.
		\begin{lemma} \label{Lemma_Bedingungen_Operator}
		If the conditions  
			\begin{align} \label{Bedingungen_Operator}
			\lim_{n \to \infty} \sup_{j \in \Z} \|A_{nj}\|_{op}=0 
			\hspace*{0.2cm} \text{ and } \hspace*{0.2cm}
			\limsup_{n \to \infty} \sum_{j \in \Z} \|A_{nj}\|_{op}^2<\infty
			\end{align}
		are fulfilled, then
			\begin{align*}
			\lim_{n \to \infty} \varrho_{3} (X_{n},Y_{n})=0.
			\end{align*}
		\end{lemma}
Referring to \citet{Gine_Leon_On_the_central_limit_theorem_in_Hilbert_space}, the processes have the same convergence behaviour if $\lim_{n \to \infty} \varrho_{3} (X_{n},Y_{n})=0$, since the metric induces the weak topology on the set of probability measures on $\Hi$.

\begin{proof} [Proof of theorem~\ref{theorem_complex_multiplication_operator_central_limit}]
We start with the main part of the proof. At the same time this is the part, which completely differs from the real-valued case.
We calculate the limit behaviour of the cross-covariances of the partial sums $S_{n}$.\\
\begin{lemma} \label{Lemma_complex_operator_limit_E(SnSn)}
If $h(r) \Intervall$ and $h(s) \Intervall$, then
	\begin{align*}
	\lim_{n \to \infty} n^{d(r,s)-3}E(S_{n}(r)\overline{S_{n}(s)}) 
	=
	\frac{c(r,s)+\overline{c(s,r)}}{(3-d(r,s))(2-d(r,s))} \sigma(r,s).
	\end{align*}
\end{lemma}
	\begin{proof}
	Changing the order of summation, we have
	\begin{align*}
	E(S_{n}(r)\overline{S_{n}(s)})
	=n\gamma_{0}(r,s)+\sum_{k=1}^{n}(n-k)\gamma_{k}(r,s)+\sum_{k=1}^{n}(n-k)\overline{\gamma_{k}(s,r)}.
	\end{align*}
	Adding the normalization, we calculate the limit for each summand. First
	\begin{align*}
	\lim_{n \to \infty} n^{d(r,s)-3}n\gamma_{0}(r,s)=0,
	\end{align*}
	since $|n^{d(r,s)-2}| \leq  n^{h(r,s)-2}$ and $h(r,s)-2 \in (-1,0)$.\\
	To prove convergence of the second summand, we calculate the Laplace transform $\mathcal{L}$ of the power function $f:[0,\infty) \to \C, t \mapsto t^{a-1}$.
	\begin{align*}
	\mathcal{L}(t^{a-1})(s)=\int_{0}^{\infty} t^{a-1} e^{-st}dt=\frac{\Gamma(a)}{s^{a}}
	\hspace*{0.2cm} \text{ for } \hspace*{0.2 cm} 
	\operatorname{Re}(s)>0 \text{ and } \operatorname{Re}(a)>0.
	\end{align*}
	Rearranging the terms, we get
	\begin{align}	\label{equality_Laplace_transform}
	s^{-a}=\frac{1}{\Gamma(a)}\int_{0}^{\infty} t^{a-1} e^{-st}dt
	\hspace*{0.2cm} \text{ for } \hspace*{0.2 cm} 
	\operatorname{Re}(s)>0 \text{ and } \operatorname{Re}(a)>0.
	\end{align}
	Combining (\ref{equality_autocovariance}) and (\ref{equality_Laplace_transform}), and applying (\ref{equality_generalized_geometrical_series}) we have
	\begin{align*}
		&n^{d(r,s)-3}\sum_{k=1}^{n}(n-k)\gamma_{k}(r,s)\\
	=&	\frac{n^{d(r,s)-3}\sigma(r,s)}{\Gamma(d(r))\Gamma(\overline{d(s)})}\sum_{k=1}^{n}(n-k)
		\int_{0}^{\infty}\int_{0}^{\infty}\sum_{j=0}^{\infty}e^{-(j+1)(x_{1}+x_{2})}e^{-kx_{2}}
		x_{1}^{d(r)-1}x_{2}^{\overline{d(s)}-1}dx_{1}dx_{2} \\
	=&	\frac{n^{d(r,s)-3}\sigma(r,s)}{\Gamma(d(r))\Gamma(\overline{d(s)})}
		\int_{0}^{\infty}\int_{0}^{\infty}\frac{1}{e^{x_{1}+x_{2}}-1}\sum_{k=1}^{n}(n-k)e^{-kx_{2}}
		x_{1}^{d(r)-1}x_{2}^{\overline{d(s)}-1}dx_{1}dx_{2}	\\
	=&	\frac{n^{d(r,s)-3}\sigma(r,s)}{\Gamma(d(r))\Gamma(\overline{d(s)})}
		\int_{0}^{\infty}\int_{0}^{\infty}\frac{1}{e^{x_{1}+x_{2}}-1}
		\frac{(n-1)e^{x_{2}}+e^{(1-n)x_{2}}-n}{(e^{x_{2}}-1)^2}
		x_{1}^{d(r)-1}x_{2}^{\overline{d(s)}-1}dx_{1}dx_{2}.
\intertext{Next, we substitute $nx_{1}=t_{1}$, $nx_{2}=t_{2}$ and take the limit by the dominated convergence theorem.}
	=&	\frac{\sigma(r,s)}{\Gamma(d(r))\Gamma(\overline{d(s)})}
		\int_{0}^{\infty}\int_{0}^{\infty}
		\frac{(n-1)e^{\frac{t_{2}}{n}}+e^{(1-n)\frac{t_{2}}{n}}-n}
		{n^{3} (e^{\frac{t_{2}}{n}}-1)^2 \left( e^{\frac{t_{1}+t_{2}}{n}}-1 \right) }
		t_{1}^{d(r)-1} t_{2}^{\overline{d(s)}-1}dt_{1}dt_{2}	\\
\overset{n \to \infty}{\longrightarrow}&
		\frac{\sigma(r,s)}{\Gamma(d(r))\Gamma(\overline{d(s)})}
		\int_{0}^{\infty}\int_{0}^{\infty}
		\frac{t_{2}+e^{-t_{2}}-1}{t_{2}^2(t_{1}+t_{2})}
		t_{1}^{d(r)-1} t_{2}^{\overline{d(s)}-1}dt_{1}dt_{2} \\
	=&	\frac{\sigma(r,s)}{\Gamma(d(r))\Gamma(\overline{d(s)})}
		\int_{0}^{\infty} x^{d(r)-1} (x+1)^{-1} dx
		\int_{0}^{\infty} (t_{2}+e^{-t_{2}}-1) t_{2}^{d(r)+\overline{d(s)}-4}dt_{2}	\\
	=&	\frac{\sigma(r,s)B(d(r),1-d(r))}{\Gamma(d(r))\Gamma(\overline{d(s)})}
		\frac{1}{3-d(r,s)}\frac{1}{2-d(r,s)}\Gamma(d(r,s)-1)
	=	\frac{\sigma(r,s)c(r,s)}{(3-d(r,s))(2-d(r,s))}.
	\end{align*}
After a repeated substitution, in this case with $t_{1}=xt_{2}$ and integration by parts twice of the second integral we applied (\ref{equality_Laplace_transform}) again.
The last step then follows from (\ref{equality_Beta_function_alternative}).	\\
Analogously we obtain the limit of the third summand by interchanging $r$ and $s$ and complex conjunction.\\
Finally, we verify that the dominated convergence theorem is applicable.
We have to prove the existence of an integrable function $g(t_{1},t_{2})$, which fulfills
	\begin{align*}
	\left\vert \frac{(n-1)e^{\frac{t_{2}}{n}}+e^{(1-n)\frac{t_{2}}{n}}-n}
		{n^{3} (e^{\frac{t_{2}}{n}}-1)^2 \left( e^{\frac{t_{1}+t_{2}}{n}}-1 \right) }
		t_{1}^{d(r)-1} t_{2}^{\overline{d(s)}-1}\right \vert
		\leq g(t_{1},t_{2})
	\end{align*}
for each $n \geq 1$. The numerator is obviously non-negative. The same holds for the denominator, since the function $n(1-e^{-\frac{t_{2}}{n}})$ monotone in $n$
	\begin{align*}
	1-e^{-t_{2}} \leq n(1-e^{-\frac{t_{2}}{n}})
	\end{align*}
holds, it follows
	\begin{align*}
	(n-1)e^{\frac{t_{2}}{n}}+e^{(1-n)\frac{t_{2}}{n}}-n \geq 0.
	\end{align*}
Using the inequalities $e^{-x} \geq 1-x$ and $x \leq e^x-1$ for each value $x \in \R$, we obtain
	\begin{align*}
	&\frac{(n-1)e^{\frac{t_{2}}{n}}+e^{(1-n)\frac{t_{2}}{n}}-n}
		{n^{3} (e^{\frac{t_{2}}{n}}-1)^2 \left( e^{\frac{t_{1}+t_{2}}{n}}-1 \right) }	
	\leq
	\frac{e^{\frac{t_{2}}{n}}}{n^{2}(e^{\frac{t_{2}}{n}}-1)^2}\frac{t_{2}-1+e^{-t_{2}}}
		{n\left( e^{\frac{t_{1}+t_{2}}{n}}-1 \right) }	\\
	\leq&
	\frac{e^{\frac{t_{2}}{n}}}{n^{2}(e^{\frac{t_{2}}{n}}-1)^2}\frac{t_{2}-1+e^{-t_{2}}}
		{t_{1}+t_{2}}
	\leq
	\frac{t_{2}-1+e^{-t_{2}}}
		{t_{2}^2(t_{2}+t_{1})}.
	\end{align*}
In the final step. we have applied the inequality
	\begin{align*}
	\frac{e^{\frac{x}{n}}}{n^2(e^{\frac{x}{n}}-1)^2}\leq \frac{1}{x^2}.
	\end{align*}
Using the power series representation we prove the inequality $e^{\frac{x}{2n}}\frac{x}{n} \leq e^{\frac{x}{n}}-1$ by
	\begin{align} \label{inequality_power_series_rep}
	e^{\frac{x}{2n}}\frac{x}{n}
	=\sum_{k=0}^{\infty} \frac{(\frac{x}{2n})^{k}}{k!}\frac{x}{n}
	=\sum_{j=1}^{\infty} \frac{\frac{j}{2^{j-1}}(\frac{x}{n})^{j}}{j!}
	\leq \sum_{j=1}^{\infty} \frac{(\frac{x}{n})^{j}}{j!}
	= e^{\frac{x}{n}}-1.
	\end{align}
All in all we obtain
	\begin{align*}
	\left\vert \frac{(n-1)e^{\frac{t_{2}}{n}}+e^{(1-n)\frac{t_{2}}{n}}-n}
		{n^{3} (e^{\frac{t_{2}}{n}}-1)^2 \left( e^{\frac{t_{1}+t_{2}}{n}}-1 \right) }
		t_{1}^{d(r)-1} t_{2}^{\overline{d(s)}-1}\right \vert
		\leq 
		\frac{t_{2}-1+e^{-t_{2}}}
		{t_{2}^2(t_{2}+t_{1})}t_{1}^{h(r)-1} t_{2}^{\overline{h(s)}-1}=:g(t_{1},t_{2})
	\end{align*}
for each $n \geq 1$. The integrability follows from the rewritten limit above.
	\end{proof}
To prove the weak convergence we use theorem \ref{theorem_weak convergence_Cremers_Kadelka}.
For the first point, we rewrite the process in form of (\ref{Prozessdarstellung_für_Beweis}). 
Denote
	\begin{align*}	v_{j}(s)=
		\begin{cases}
		(j+1)^{-d(s)}	&\text{, if } j\geq0	\\
		0				&\text{, if } j<0
		\end{cases}
	\end{align*}
and rewrite the partial sums as
	\begin{align}
	S_{n}(s)
	= \sum_{k=1}^{n} \sum_{j=0}^{\infty} v_{j}(s) \varepsilon_{k-j}(s)
	=\sum_{j=-\infty}^{n +1} a_{nj}(s) \varepsilon_{j}(s) \label{Darstellung_Partialsummenprozess}
	\end{align}
with $a_{nj}(s)=\sum_{k=1}^{n} v_{k-j}(s)$. We define $z_{n}(s)=n^{d(s)-\frac{3}{2}}$ and we consider for $s_{1},...,s_{p} \in \Sg$ the sequence of random vectors
	\begin{align*}
	(n^{-H}S_{n}(s_{1}),...,n^{-H}S_{n}(s_{p}))^t=\sum_{j=-\infty}^{\infty} A_{nj} \varepsilon_{j}
	\end{align*}
with $A_{nj}=\operatorname{diag}(z_{n}^{-1}(s_{1})a_{nj}(s_{1}),...,z_{n}^{-1}(s_{p})a_{nj}(s_{p}))$ and $\varepsilon_{j}=(\varepsilon_{j}(s_{1}),...,\varepsilon_{j}(s_{p}))^{t}$.
\begin{lemma}
If $h(s) \Intervall$, the sequence of operators $A_{nj}$ fulfills the conditions (\ref{Bedingungen_Operator}).
\end{lemma}
	\begin{proof}
	Since $\|A_{nj}\|_{op}=\max_{1\leq i \leq q} | n^{d(s_{i})-\frac{3}{2}}a_{nj}(s_{i})|$, first using the triangle inequality
		\begin{align*}
		 \sup_{j \in \Z}	| n^{d(s)-\frac{3}{2}}a_{nj}(s)|
		\leq \sup_{j \in \Z}	n^{h(s)-\frac{3}{2}}\sum_{k=1}^{n}|v_{k-j}(s)|
		\sim	 n^{h(s)-\frac{3}{2}}\frac{n^{1-h(s)}}{1-h(s)}.
		\end{align*}
Second using the relationship between the variances of the partial sums and the sequence $a_{nj}(s)$
		\begin{align*}
		  E|S_{n}(s)|^2
		=\sigma^2(s)\sum_{j=-\infty}^{n +1} |a_{nj}(s)|^2,
		\end{align*}
	 we obtain
		\begin{align*}
		\sum_{j\in \Z} \|A_{nj}\|_{op}^2
		=\sum_{j\in \Z}| n^{d(s)-\frac{3}{2}}a_{nj}(s)|^2
		=n^{2h(s)-3}\sum_{j\in \Z}|a_{nj}(s)|^2
		=n^{2h(s)-3}\frac{1}{\sigma^2(s)}E|S_{n}(s)|^2
		\end{align*}
	and using Lemma \ref{Lemma_complex_operator_limit_E(SnSn)}
		\begin{align*}
		\lim_{n \to \infty} n^{2h(s)-3}\frac{1}{\sigma^2(s)}E|S_{n}(s)|^2
		=&\frac{\overline{c(s)}+c(s)}{(3-2h(s))(2-2h(s))}.
		\end{align*}
	\end{proof}
It remains to show the convergence of the process
		\begin{align*}
		(n^{-H}\tilde{S}_{n}(s_{1}),...,n^{-H}\tilde{S}_{n}(s_{p}))
		=\sum_{j=-\infty}^{\infty} A_{nj}\tilde{\varepsilon}_{j},
		\end{align*}
which follows from Lemma \ref{Lemma_complex_operator_limit_E(SnSn)}.
\vspace*{0.2cm}\\
Setting $r=s$ in Lemma \ref{Lemma_complex_operator_limit_E(SnSn)}, we get the proof of point (II) in theorem \ref{theorem_weak convergence_Cremers_Kadelka}.\\
The following lemma provides the proof of point (III).
\begin{lemma} \label{lemma_dominated convergence_integrable function_central_limt}
There exists a $\mu$-integrable function $g$ with $|E|n^{-H}S_{n}(s)|^2| \leq g(s)$ for each $s$ in $\Sg$ and $n \in \N$, if the integrals
	\begin{align*}
	\int_{\Sg} \frac{\sigma^2(s)}{(1-h(s))^2}\mu(ds)
	\hspace*{0.2cm} \text{ and } \hspace*{0.2cm}
	\int_{\Sg} \frac{\sigma^2(s)}{(1-h(s))(2h(s)-1)}\mu(ds)
	\end{align*}	
are finite.
\end{lemma}
	\begin{proof}
	Rewriting the variances and using the triangular inequality, we get
		\begin{align*}
			|E|n^{-H}S_{n}(s)|^2|
		\leq n^{2h(s)-3}
			\left( n|\gamma_{0}(s)|+\sum_{k=1}^{n}(n-k)\left| \gamma_{k}(s)\right| +\sum_{k=1}^{n}(n-k)\left| \overline{\gamma_{k}(s)}\right| \right) .
		\end{align*}
		The first two summands are approxiamable as
		\begin{align*}
			n^{2h(s)-3} n|\gamma_{0}(s)|
		\leq& \sigma^2(s)\left( 1 + \frac{1}{2h(s)-1}\right)	\\
	\intertext{and}
			n^{2h(s)-3} \sum_{k=1}^{n}(n-k)\left| \gamma_{k}(s)\right|
		\leq&	n^{2h(s)-3} \sigma^2(s)\sum_{k=1}^{n}(n-k) k^{1-2h(s)}c_{h}(s)	\\
		\leq& \sigma^2(s)\frac{1}{2}\left( \frac{1}{(1-h(s))^2} + \frac{1}{(1-h(s))(2h(s)-1)}\right),
		\end{align*}
where the last inequality follows from (\ref{inequality_c(s)_in_dependence_of_h}).
The approximation of the third one works analogously and altogether we obtain
		\begin{align*}
		|E|n^{-H}S_{n}(s)|^2|
		\leq 	2\left( \frac{\sigma^2(s)}{(1-h(s))^2} + \frac{\sigma^2(s)}{(1-h(s))(2h(s)-1)}\right)=:g(s).
		\end{align*}
	\end{proof}	
The proof of the central limit theorem is completed.
	\end{proof}

	\begin{proof} [Proof of theorem~\ref{theorem_complex_multiplication_operator_functional_central_limit}]
	\allowdisplaybreaks
	We are interested in the convergence behaviour of the piecewise linear process $(\zeta_{n}(t))_{ n \in \N }$ with
	\begin{align*}
	\zeta_{n}(t)=S_{\lfloor nt \rfloor} + \{nt\} X_{ \lfloor nt \rfloor +1}
	\hspace*{0.2cm} \text{ for } \hspace{0.2cm} t\in [0,1].
	\end{align*}
With reference to \citet{Chara_Rack_Operator} the expression may be represented as a linear combination of a sequence of operators and the random process $(\varepsilon_{k})_{ k \in \Z }$. Let
		\begin{align*}	v_{j}(s)=
		\begin{cases}
		(j+1)^{-d(s)}	&\text{, if } j\geq0	\\
		0				&\text{, if } j<0
		\end{cases},
	\end{align*}
then
	\begin{align*}
	\zeta_{n}(s,t)
	=S_{\lfloor nt \rfloor}(s) + \{nt\} X_{\lfloor nt \rfloor +1}(s)
	=	\sum_{j=-\infty}^{\lfloor nt \rfloor +1} a_{nj}(s,t) \varepsilon_{j}(s)
	\end{align*}
with
		\begin{align} \label{anj}
		a_{nj}(s,t)=\sum_{k=1}^{\lfloor nt \rfloor} v_{k-j}(s) + \{nt\} v_{\lfloor nt \rfloor +1-j}(s).
		\end{align}

Next we show the relation between the limit of the normalised cross-covariances of the piecewise linear process $\zeta_{n}$ and of the normalised cross-covariances of the partial sums $S_{n}$.
\begin{lemma} \label{Lemma_limE(zetazeta)=limE(SnSn)}
If $h(s) \Intervall$ and $h(r) \Intervall$, then
	\begin{align*}
	\lim_{n \to \infty}n^{d(r,s)-3}E(\zeta_{n}(r,t)\overline{\zeta_{n}(s,u)})
	=\lim_{n \to \infty}n^{d(r,s)-3}E(S_{\lfloor nt \rfloor}(r)\overline{S_{\lfloor nu \rfloor}(s)}).
	\end{align*}
\end{lemma}
\begin{lemma} \label{Konvergenz_Kovarianzen_Invarianzpr_Multi_komplex}
If $h(s) \Intervall$ and $h(r) \Intervall$, then
	\begin{align*}
	\lim_{n \to \infty}n^{d(r,s)-3}E(S_{\lfloor nt \rfloor}(r)\overline{S_{\lfloor nu \rfloor}(s)}) 
	=V((r,t),(s,u))
	\end{align*}
for $(r,t),(s,u) \in \Sg \times [0,1]$, where $V$ is given in (\ref{equality_function_V}).
\end{lemma}
	\begin{proof} We assume $t<u$ and supplement $\overline{S_{\lfloor nt \rfloor}(s)}$ in the expression of the cross-covariances.
		\begin{align*}
		E(S_{\lfloor nt \rfloor}(r)\overline{S_{\lfloor nu \rfloor}(s)})
		=E(S_{\lfloor nt \rfloor}(r)\overline{S_{\lfloor nt \rfloor}(s)}) + 
		E(S_{\lfloor nt \rfloor}(r)(\overline{S_{\lfloor nu \rfloor}(s)-S_{\lfloor nt \rfloor}(s)}) 
		\end{align*}
	Using Lemma \ref{Lemma_complex_operator_limit_E(SnSn)}, we are able to calculate the limit of the first summand as
		\begin{align}	\label{Gleichung aus ZGWS}
		\lim_{n \to \infty} n^{d(r,s)-3}E(S_{\lfloor nt \rfloor}(r)\overline{S_{\lfloor nt \rfloor}(s)})
		=
		\frac{c(r,s)+\overline{c(s,r)}}{(3-d(r,s))(2-d(r,s))} \sigma(r,s)t^{3-d(r,s)}.
		\end{align}
The second therm is may be separated into three sums, where $m_{n}=\min(\lfloor nt \rfloor,\lfloor nu \rfloor-\lfloor nt \rfloor)$.
		\begin{align*}
		&E(S_{\lfloor nt \rfloor}(r)(\overline{S_{\lfloor nu \rfloor}(s)-S_{\lfloor nt \rfloor}(s)})\\
		=& 	\sum_{k=1}^{m_{n}-1}k\gamma_{k}(r,s)
			+ m_{n} \sum_{k=m_{n}}^{\lfloor nu \rfloor -m_{n}}\gamma_{k}(r,s)
			+ \sum_{k=\lfloor nu \rfloor-m_{n}+1}^{\lfloor nu \rfloor-1}(\lfloor nu \rfloor-k)\gamma_{k}(r,s)
		\end{align*}
We remind of the representation of the autocovariance function, we already used in the proof of Lemma \ref{Lemma_complex_operator_limit_E(SnSn)} by using the Laplace transform.
	\begin{align*}
	\gamma_{k}(r,s)=\int_{0}^{\infty}\int_{0}^{\infty}\frac{1}{e^{x_{1}+x_{2}}-1}e^{-kx_{2}}
		x_{1}^{d(r)-1}x_{2}^{\overline{d(s)}-1}dx_{1}dx_{2}
	\end{align*}
With the help of (\ref{equality_generalized_geometrical_series}) we get three expressions:
	\begin{align*}
		G^1_{n}(x_{2})
	:=	& \sum_{k=1}^{m_{n}-1}ke^{-kx_{2}}
	=	\frac{e^{x_{2}}-m_{n}e^{(2-m_{n})x_{2}}-(1-m_{n})e^{(1-m_{n})x_{2}}}{(e^{x_{2}}-1)^2}	\\
		G^2_{n}(x_{2})
	:=&	m_{n}\sum_{k=m_{n}}^{\lfloor nu \rfloor - m_{n}}e^{-kx_{2}}
	=	m_{n}\frac{e^{(2-m_{n})x_{2}}-e^{(1-(\lfloor nu \rfloor-m_{n}))x_{2}}-e^{(1-m_{n})x_{2}}+e^{-(\lfloor nu \rfloor-m_{n})x_{2}}}{(e^{x_{2}}-1)^2}	\\
		G^3_{n}(x_{2})
	:=&	\sum_{k=\lfloor nu \rfloor-m_{n}+1}^{\lfloor nu \rfloor-1}(\lfloor nu \rfloor-k)e^{-kx_{2}}	\\
	=&	\frac{(m_{n}-1)e^{(m_{n}-\lfloor nu \rfloor+1)x_{2}}-m_{n}e^{(m_{n}-\lfloor nu \rfloor)x_{2}}+e^{(1-\lfloor nu \rfloor x_{2})}}{(e^{x_{2}}-1)^2}
	\end{align*}
	Combining our results, adding the normalization sequence and substituting $x_{1}=\frac{t_{1}}{n}$ and $x_{2}=\frac{t_{2}}{n}$, we obtain
	\begin{align*}
	& \frac{n^{d(r,s)-3} \sigma(r,s)}{\Gamma(d(r))\Gamma(\overline{d(s)})}
		\int_{0}^{\infty}\int_{0}^{\infty}\frac{1}{e^{x_{1}+x_{2}}-1}
		\left( G^1_{n}(x_{2}) +
			G^2_{n}(x_{2})  +
			G^3_{n}(x_{2}) \right)
		x_{1}^{d(r)-1}x_{2}^{\overline{d(s)}-1}dx_{1}dx_{2}	\\
	=&n^{d(r,s)-3}\frac{\sigma(r,s)}{\Gamma(d(r))\Gamma(\overline{d(s)})}
			\int_{0}^{\infty}\int_{0}^{\infty}
			\frac{e^{\frac{t_{2}}{n}}-e^{(1-m_{n})\frac{t_{2}}{n}}-e^{(m_{n}-\lfloor nu\rfloor+1)\frac{t_{2}}{n}}+e^{(1-\lfloor nu\rfloor)\frac{t_{2}}{n}}}{(e^{\frac{t_{2}}{n}}-1)^2} \cdot\\
			& \phantom{n^{d(r,s)-3}\frac{\sigma(r,s)}{\Gamma(d(r))\Gamma(\overline{d(s)})}
			\int_{0}^{\infty}\int_{0}^{\infty}}
			\frac{1}{  e^{\frac{t_{1}+t_{2}}{n}}-1} 
			\left( \frac{t_{1}}{n}\right)^{d(r)-1} \left( \frac{t_{2}}{n}\right) ^{\overline{d(s)}-1}n^{-2}dt_{1}dt_{2}\\
	=&\frac{\sigma(r,s)}{\Gamma(d(r))\Gamma(\overline{d(s)})}
			\int_{0}^{\infty}\int_{0}^{\infty}
			\frac{e^{\frac{t_{2}}{n}}-e^{(1-m_{n})\frac{t_{2}}{n}}-e^{(m_{n}-\lfloor nu\rfloor+1)\frac{t_{2}}{n}}+e^{(1-\lfloor nu\rfloor)\frac{t_{2}}{n}}}{n^3 (e^{\frac{t_{2}}{n}}-1)^2(e^{\frac{t_{1}+t_{2}}{n}}-1)}
			t_{1}^{d(r)-1} t_{2}^{\overline{d(s)}-1}dt_{1}dt_{2}.
	\intertext{Let $m=\min(t,u-t)$, then we use the dominated convergence theorem and substitute with $t_{1}=x t_{2}$.}
	\overset{n \to \infty}{\longrightarrow}&
	\frac{\sigma(r,s)}{\Gamma(d(r))\Gamma(\overline{d(s)})}
			\int_{0}^{\infty}\int_{0}^{\infty}
			(1-e^{-mt_{2}}-e^{-(u-m)t_{2}}+e^{-ut_{2}})\frac{1}{t_{2}^{2}(t_{1}+t_{2})}
			t_{1}^{d(r)-1} t_{2}^{\overline{d(s)}-1}dt_{1}dt_{2}	\\
	&=	\frac{\sigma(r,s)B(d(r),1-d(r))}{\Gamma(d(r))\Gamma(\overline{d(s)})}
			\int_{0}^{\infty}
			(1-e^{-mt_{2}}-e^{-(u-m)t_{2}}+e^{-ut_{2}})
			t_{2}^{\overline{d(s)}+d(r)-4}dt_{2}	\\
	&=	c(r,s)\frac{\sigma(r,s)}{(3-d(r,s))(2-d(r,s))}\left( -m^{3-d(r,s)}-(u-m)^{3-d(r,s)}+u^{3-d(r,s)} \right).
	\end{align*}
In the last step we used integration twice by parts. 
Combining the limits and (\ref{Gleichung aus ZGWS}), we get
	\begin{align*}
	&\lim_{n \to \infty}
		n^{3-d(r,s)}E(S_{\lfloor nt \rfloor}(r)\overline{S_{\lfloor nu \rfloor}(s)}) \\
	=	&\frac{\sigma(r,s)}{(3-d(r,s))(2-d(r,s))}\\
		&\left[(c(r,s)+\overline{c(s,r)})t^{3-d(r,s)}+c(r,s)(-m^{3-d(r,s)}+u^{3-d(r,s)}-(u-m)^{3-d(r,s)}) \right]	\\
	=	&\frac{\sigma(r,s)}{(3-d(r,s))(2-d(r,s))}
		\left[\overline{c(s,r)}t^{3-d(r,s)}+c(r,s)(u^{3-d(r,s)} -(u-t)^{3-d(r,s)}) \right]
	\end{align*}
in case of $t<u$. Analogously for $t>u$
	\begin{align*}
	\lim_{n \to \infty}
		&n^{3-d(r,s)}E(S_{\lfloor nt \rfloor}(r)\overline{S_{\lfloor nu \rfloor}(s)}) \\
	=	&\frac{\sigma(r,s)}{(3-d(r,s))(2-d(r,s))}\left[c(r,s)u^{3-d(r,s)}+\overline{c(s,r)}(t^{3-d(r,s)} -(t-u)^{3-d(r,s)}) \right]
	\end{align*}
follows. Both cases together provide
	\begin{align*}
	\lim_{n \to \infty}
		&n^{3-d(r,s)}E(S_{\lfloor nt \rfloor}(r)\overline{S_{\lfloor nu \rfloor}(s)}) \\
	=	&\frac{\sigma(r,s)}{(3-d(r,s))(2-d(r,s))}
		\left[\overline{c(s,r)}t^{3-d(r,s)}+c(r,s)u^{3-d(r,s)}+C(r,s;t-u)|t-u|^{3-d(r,s)}\right].
	\end{align*}
We still need to prove the interchangeability of the integrals and the limit, so, using the dominated convergence theorem, we have to show the existence of an integrable function $g: \R_{\geq 0} \times \R_{\geq 0} \to \R_{\geq 0}$, such that
	\begin{align*}
	\left\vert  \frac{e^{\frac{t_{2}}{n}}-e^{(1-m_{n})\frac{t_{2}}{n}}-e^{(m_{n}-\lfloor nu\rfloor+1)\frac{t_{2}}{n}}
	+e^{(1-\lfloor nu\rfloor)\frac{t_{2}}{n}}}{n^3 (e^{\frac{t_{2}}{n}}-1)^2\left( e^{\frac{t_{1}+t_{2}}{n}}-1\right)}t_{1}^{d(r)-1} t_{2}^{\overline{d(s)}-1} \right\vert
	\leq
	g(t_{1},t_{2})
	\end{align*}
for each $n \geq 1$. Under the assumption $t<u$ the numerator is obviously non-negative. The denominator is positive, since
by case analysis in dependence of $m_{n}$	
	\begin{align*}
	&e^{\frac{t_{2}}{n}}-e^{(1-m_{n})\frac{t_{2}}{n}}-e^{(m_{n}-\lfloor nu\rfloor+1)\frac{t_{2}}{n}}
	+e^{(1-\lfloor nu\rfloor)\frac{t_{2}}{n}} \geq 0	\\
	\Leftrightarrow &
	0 \geq m_{n}-\lfloor nu\rfloor	\\
	\Leftrightarrow&	
	0 \geq
	\begin{cases}
	\lfloor nt\rfloor-\lfloor nu\rfloor	&\hspace*{0.2cm} \text{, for } m_{n}=\lfloor nt \rfloor\\
	-\lfloor nt\rfloor 							&\hspace*{0.2cm} \text{, for } m_{n}=\lfloor nu\rfloor-\lfloor nt \rfloor
	\end{cases},
	\end{align*}
which is fulfilled under the assumption $t<u$. So using (\ref{inequality_power_series_rep}), we obtain
	\begin{align*}
	& \left\vert  \frac{e^{\frac{t_{2}}{n}}-e^{(1-m_{n})\frac{t_{2}}{n}}-e^{(m_{n}-\lfloor nu\rfloor+1)\frac{t_{2}}{n}}
	+e^{(1-\lfloor nu\rfloor)\frac{t_{2}}{n}}}{n^3 (e^{\frac{t_{2}}{n}}-1)^2\left( e^{\frac{t_{1}+t_{2}}{n}}-1\right)}t_{1}^{d(r)-1} t_{2}^{\overline{d(s)}-1} \right\vert	\\
=&
	\frac{e^{\frac{t_{2}}{n}}}{n^2 (e^{\frac{t_{2}}{n}}-1)^2}\frac{1}{n\left( e^{\frac{t_{1}+t_{2}}{n}}-1\right)}(1-e^{-m_{n}\frac{t_{2}}{n}}-e^{(m_{n}-\lfloor nu\rfloor)\frac{t_{2}}{n}}
	+e^{-\lfloor nu\rfloor\frac{t_{2}}{n}})t_{1}^{h(r)-1} t_{2}^{h(s)-1}	\\
\leq&
	\frac{1-e^{-m t_{2}}-e^{(m-u) t_{2}}+e^{-u t_{2}}}{t_{2}^2(t_{2}+t_{1})}t_{1}^{h(r)-1} t_{2}^{h(s)-1}=:g(t_{1},t_{2})
	\end{align*}
for each $n \geq 1$. In the last step, we applied the inequality
	\begin{align*}
	1-e^{-m_{n}\frac{t_{2}}{n}}-e^{(m_{n}-\lfloor nu\rfloor)\frac{t_{2}}{n}}
	+e^{-\lfloor nu\rfloor\frac{t_{2}}{n}} \leq 1-e^{-m t_{2}}-e^{(m-u) t_{2}}+e^{-u t_{2}},
	\end{align*}
which holds, since
\begin{alignat*}{2} 
					&&	\frac{e^{m_{n}\frac{t_{2}}{n}}-1}{e^{m t_{2}}-1} (1-e^{(m_{n}-\lfloor nu\rfloor+mn)\frac{t_{2}}{n}})
						\leq &
						e^{m_{n}\frac{t_{2}}{n}-m t_{2}} (1-e^{(m_{n}-\lfloor nu\rfloor+mn)\frac{t_{2}}{n}})	\\
					&&	=& 
						e^{m_{n}\frac{t_{2}}{n}-m t_{2}} -e^{(2m_{n}-\lfloor nu\rfloor)\frac{t_{2}}{n}}\\
					&&	\leq &
						1-e^{(m_{n}-un+mn)\frac{t_{2}}{n}}	\\
\Rightarrow &&	1-e^{-m_{n}\frac{t_{2}}{n}}-e^{(m_{n}-\lfloor nu\rfloor)\frac{t_{2}}{n}}+e^{-\lfloor nu\rfloor\frac{t_{2}}{n}} 
		\leq & 1-e^{-m t_{2}}-e^{(m-u) t_{2}}+e^{-u t_{2}}	\\
\Leftrightarrow && (e^{m_{n}\frac{t_{2}}{n}}-1)(1-e^{(m_{n}-\lfloor nu\rfloor+mn)\frac{t_{2}}{n}}) 
		\leq & (e^{m t_{2}}-1)(1-e^{(m_{n}-un+mn)\frac{t_{2}}{n}}).
\end{alignat*} 
where the last inequality follows from $m_{n}\frac{t_{2}}{n}-m t_{2} \leq 0$ and $2m_{n}-\lfloor nu\rfloor \geq m_{n}-un+mn$.
The function $g$ is integrable, because the calculated limit above is finite.
	\end{proof}

	\begin{proof} [Proof of Lemma~\ref{Lemma_limE(zetazeta)=limE(SnSn)}]
	It holds
		\begin{align*}
			E(\zeta_{n}(r,t)\overline{\zeta_{n}(s,u)})
		=	&E(S_{\lfloor nt \rfloor}(r)\overline{S_{\lfloor nu \rfloor}(s)})+
			\{nu\}E(S_{\lfloor nt \rfloor}(r)\overline{X_{\lfloor nu \rfloor+1}(s)})+\\
			&\{nt\}E(S_{\lfloor nu \rfloor}(s)\overline{X_{\lfloor nt \rfloor+1}(r)})+
			\{nu\}\{nt\} E(X_{\lfloor nu \rfloor+1}(r)\overline{X_{\lfloor nt \rfloor+1}(s)}).
		\end{align*}
	We have to show that
		\begin{align*}
		\lim_{n \to \infty}n^{d(r,s)-3}
		&\Big(\{nu\}E(S_{\lfloor nt \rfloor}(r)\overline{X_{\lfloor nu \rfloor+1}(s)})+
		\{nt\}E(S_{\lfloor nu \rfloor}(s)\overline{X_{\lfloor nt \rfloor+1}(r)})+\\
		&\{nu\}\{nt\} E(X_{\lfloor nu \rfloor+1}(r)\overline{X_{\lfloor nt \rfloor+1}(s)})\Big)=0.
		\end{align*}
	The first summand could be estimated as
		\begin{align*}
		\Big| n^{d(r,s)-3}\{nu\}E(S_{\lfloor nt \rfloor}(r)\overline{X_{\lfloor nu \rfloor+1}(s)})\Big|
		\leq \frac{1}{n^{2-h(r,s)}} \gamma_{0}(r,s),
		\end{align*}
	the convergence to zero of the other summands follows analogously.
	\end{proof}
Combining the Lemmas \ref{Lemma_limE(zetazeta)=limE(SnSn)} and \ref{Konvergenz_Kovarianzen_Invarianzpr_Multi_komplex} yields
	\begin{align} \label{equality_Convergence_behaviour_piecewise_linear_process}
	\lim_{n \to \infty}n^{3-d(r,s)}E(\zeta_{n}(r,t)\overline{\zeta_{n}(s,u)}) 
	=V((r,t),(s,u)).
	\end{align}
Next we show the existence of a Gaussian process $\mathcal{G}$ with zero-mean and cross-covariance function $V$. Additionally we will prove that the process is operator self-similar.
\begin{lemma} \label{Lemma_hermitesch_positiv_definit}
The function $V:\mathbb{T} \times \mathbb{T}\rightarrow \C$ defined in (\ref{equality_function_V}) is hermitian and non-negative definit. 
\end{lemma}
	\begin{proof} Clearly $V$ is hermitian, since $\sigma(r,s)$ is hermitian.\\
	Let $N\in \N$, $\tau_{1},...,\tau_{N} \in  \mathbb{T}$, $w_{1},...,w_{N} \in \C$ and $M:=\max\{t_{1},...,t_{N}\}$ then
		\begin{align*}
			&\sum_{i=1}^{N}\sum_{j=1}^{N} w_{i}\overline{w_{j}}V(\tau_{i},\tau_{j}) \\
		=&	\sum_{i=1}^{N}\sum_{j=1}^{N} w_{i} M^{\frac{3}{2}-d(s_{i})} \overline{w_{j}} M^{\frac{3}{2}-\overline{d(s_{j})}} 
			V\left( \left( s_{i},\left( \frac{t_{i}}{M}\right)\right) ,\left( s_{j},\left( \frac{t_{j}}{M}\right)\right) \right).
	\intertext{We denote $\tilde{w}_{i}:=w_{i}M^{\frac{3}{2}-d(s_{i})}$ and use Lemma \ref{Konvergenz_Kovarianzen_Invarianzpr_Multi_komplex}}
		=&	\sum_{i=1}^{N}\sum_{j=1}^{N} \tilde{w}_{i} \overline{\tilde{w}_{j}} 
			\lim_{n\to \infty} \frac{1}{n^{3-d(s_{i},s_{j})}}
			E\left( \zeta_{n}\left( s_{i},\left( \frac{t_{i}}{M}\right)\right)
			\overline{\zeta_{n}\left( s_{j},\left( \frac{t_{j}}{M}\right)\right)} \right) 
		\geq 0.
		\end{align*}
	Since $\frac{t_{i}}{M}$ is in the unit interval $[0,1]$ and 
	$\frac{1}{n^{3-d(r,s)}}E(\zeta_{n}(r,s)\overline{\zeta_{n}(s,u)})$ is a covariance function for each $(r,t),(s,u) \in \Sg \times [0,1] $.
	\end{proof}
So there exists a Gaussian process $\mathcal{G}$ with covariance function $V$.
\begin{lemma} \label{Proposition 8}
If $h(s) \Intervall$ and the integrals
	\begin{align*}
	\int_{\Sg} \frac{\sigma^2(v)}{(1-h(v))^{2}}\mu(dv) \hspace{0.2cm} \text{ and } \hspace{0.2cm} 
	\int_{\Sg}\frac{\sigma^2(v)}{(1-h(v))(2h(v)-1)}\mu(dv)
	\end{align*}
are finite, then for each $t \in \R_{\geq0}$ the stochastic process $\{\mathcal{G}(s,t) | s \in \Sg \}$ has sample paths in $\mathcal{L}_{2}(\mu,\C)$. Furthermore $\mathcal{G}(\cdot,t)$ are Gaussian random elements with values in $L_{2}(\mu,\C)$ and the process $\{\mathcal{G}(\cdot,t) | t \in \R_{\geq 0} \}$ is Gaussian.
\end{lemma}
	\begin{proof}
Using the inequalities (\ref{inequality_realpart_c(s)}) and (\ref{inequality_c(s)_in_dependence_of_h}), we get the approximation
		\begin{align} \label{inequality_bound_function_functional_central_limit}
				&E\left( \int_{\Sg} |\mathcal{G}(v,t)|^2 \mu(dv)\right)
		=	\int_{\Sg} \frac{\sigma^2(v)}{2(1-h(v))(3-2h(v))}(\overline{c(v)}+c(v))t^{3-2h(v)} \mu(dv) \\
		&\leq	\max\{t,t^2\} \left( \int_{\Sg} \frac{\sigma^2(v)}{(1-h(v))^2}\mu(dv) + 
				\int_{\Sg} \frac{\sigma^2(v)}{(1-h(v))(2h(v)-1)} \mu(dv)\right)
		<\infty. \nonumber
		\end{align}
So the paths of the stochastic process $\{\mathcal{G}(s,t) | s \in \Sg \}$ are $P$-almost surely in $L_{2}(\mu,\C)$ for each $t \in \R_{\geq 0}$ and so $\mathcal{G}(\cdot,t)$ is a Gaussian random element in $L_{2}(\mu,\C)$.
	\end{proof}
\begin{lemma} \label{lemma_operator_self_similar_multi}
The stochastic process $\{\mathcal{G}(\cdot , t) | t \in [0,\infty)\}$ is operator self-similar with scaling family $\{a^{H}|a>0\}$, where $a^{H}$ with $a>0$ is given by $a^{H}f=\{a^{\frac{3}{2}-d(s)}f(s)|s\in \Sg \}$ for $f \in L_{2}(\mu,\C)$.
\end{lemma}
	\begin{proof}
	We have to prove
		\begin{align*}
		\langle E[\langle \mathcal{G}(\cdot,at),f \rangle \overline{\mathcal{G}(\cdot,au)}],g \rangle
		=\langle E[\langle a^{H} \mathcal{G}(\cdot,t),f \rangle \overline{a^{H} \mathcal{G}(\cdot,u)}],g \rangle
		\hspace*{0.2cm} \text{ for each } f,g \in L_{2}(\mu,\C).
		\end{align*}
	Since
		\begin{align*}
			E(\mathcal{G}(r,at)\overline{\mathcal{G}(s,au)})
		=	E(a^{\frac{3}{2}-d(r)}\mathcal{G}(r,t)a^{\frac{3}{2}-\overline{d(s)}}\overline{\mathcal{G}(s,u)}),
		\end{align*}
	we obtain
		\begin{align*}
			&\langle E[ \langle a^H \mathcal{G}(\cdot,t),f \rangle \overline{a^H\mathcal{G}(\cdot,u)}],g\rangle \\
		=	& \int_{\Sg} E[ \int_{\Sg} a^{\frac{3}{2}-d(r)}\mathcal{G}(r,t)\overline{f(r)}\mu(dr) 
			a^{\frac{3}{2}-\overline{d(s)}}\overline{\mathcal{G}(s,u)}]\overline{g(s)}\mu(ds) \\
		=	& \int_{\Sg} \int_{\Sg} E[ a^{\frac{3}{2}-d(r)}\mathcal{G}(r,t) 
			a^{\frac{3}{2}-\overline{d(s)}}\overline{\mathcal{G}(s,u)}]\overline{f(r)}\mu(dr)\overline{g(s)}\mu(ds)\\
		=	& \langle E[ \langle \mathcal{G}( \cdot ,at),f \rangle \overline{\mathcal{G}(s,au)}], g \rangle
		\end{align*}
	for all $f,g \in L_{2}(\mu,\C)$.
	\end{proof}
We are now able to define the limit process, which is a restriction to the unit interval of the process $\{ \mathcal{G}(s,t) | (s,t) \in \Sg \times \R_{\geq 0} \}$ given by
	\begin{align*}
	G=\{G(s,t) | (s,t) \in \Sg \times [0,1]\}.
	\end{align*}
The following lemma establishes sufficient conditions for the existence of a continuous version of this process.
\begin{lemma} \label{Proposition 10}
If the integrals
	\begin{align*}
	\int_{\Sg} \frac{\sigma^2(v)}{(1-h(v))^{2}}\mu(dv)
	\hspace{0.2cm} \text{ and } \hspace{0.2cm} 
	\int_{\Sg}\frac{\sigma^2(v)}{(1-h(v))(2h(v)-1)}\mu(dv)
	\end{align*}
are finite, there exists a continuous version of the $L_{2}(\mu,\C)$-valued stochastic process\\
$G=\{G(\cdot,t) | t \in [0,1]\}$.
\end{lemma}
	\begin{proof}
	We use Kolmogorov`s continuity theorem (see \citet{kallenberg2014foundations}) and an inequality for the moments of a Hilbert space-valued Gaussian random element (see \citet{ledoux1991probability}).
		\begin{align*}
						&E\| G(\cdot,t)-G(\cdot,u) \|^4
		\leq 		K_{4,2}^{4} (E\| G(\cdot,t)-G(\cdot,u) \|^2)^2	\\
		=		& K_{4,2}^{4} \left( \int_{\Sg} \frac{\sigma^2(v)}{(1-h(v))(3-2h(v))} \frac{1}{2}(c(v)+\overline{c(v)})
				|t-u|^{3-2h(v)} \mu(dv) \right) ^2\\
		\leq		& K_{4,2}^{4} \left( \int_{\Sg} \frac{\sigma^2(v)}{(1-h(v))^{2}}\mu(dv)+
				\int_{\Sg}\frac{\sigma^2(v)}{(1-h(v))(2h(v)-1)}\mu(dv) \right) ^2 |t-u|^2.
		\end{align*}
 We used the inequalities (\ref{inequality_realpart_c(s)}) and (\ref{inequality_c(s)_in_dependence_of_h}).
	\end{proof}

The proof is divided into the convergence of the finite-dimensional distributions and the tightness.

\subsubsection{Convergence of the finite-dimensional distributions}
We consider $\zeta_{n}$ and $G$ as random elements in $C([0,1], L_{2}(\mu,\C))$, i.e.
	\begin{align*}
	\zeta_{n},G:\Omega \rightarrow C([0,1], L_{2}(\mu,\C)).
	\end{align*}
We get as a necessary condition for the convergence of the finite-dimensional distributions
	\begin{align*}
	(n^{-H}\zeta_{n}(t_{1}),...,n^{-H}\zeta_{n}(t_{q}))
	\overset{D}{\rightarrow}
	(G(t_{1}),...,G(t_{q}))
	\end{align*}
in $L_{2}^{q}(\mu,\C)$ for each $q \in \N$ and all $t_{1},...,t_{q} \in [0,1]$. \\
The space $L_{2}^{q}(\mu,\C)$ is isomorph to the space of all $\mu$-integrable functions $f: \Sg \rightarrow \C^{q}$ denoted by $L_{2}(\mu,\C^{q})$ and endowed with the norm
	\begin{align*}
	\|f\|=\left( \int_{\Sg} \|f(s)\|^2_{\C^{q}} \mu(ds) \right) ^{\frac{1}{2}}
	\hspace{0.2cm} \text{ for } \hspace{0.2cm}	
	f \in L_{2}(\mu,\C^{q}),
	\end{align*}
where $\| \cdot \|_{\C^{q}}$ is the euclidean norm in $\C^{q}$. We define
	\begin{align*}
	\zeta_{n}^{(q)}(s)
	=(\zeta_{n}(s,t_{1}),...,\zeta_{n}(s,t_{q}))^{T}
	\hspace*{0.2cm} \text{ and } \hspace*{0.2cm}
	G^{(q)}(s)=(G(s,t_{1}),...,G(s,t_{q}))^{T}
	\end{align*}
for $s \in \Sg$ and fixed $t_{1},...,t_{q} \in [0,1]$ and the processes
$\zeta_{n}^{(q)}=\{\zeta_{n}^{(q)}(s) |s \in \Sg\}$ and $G^{(q)}=\{G^{(q)}(s) |s \in \Sg\}$. Using the isomorphy of $L_{2}^{q}(\mu,\C)$ and $L_{2}(\mu,\C^{q})$ we are able to prove equivalently 
	\begin{align} \label{equality_convergence_finite_dim_functional_central}
	n^{-H} \zeta_{n}^{(q)} \overset{D}{\rightarrow} G^{(q)}
	\end{align}
in $L_{2}(\mu,\C^{q})$.
We use theorem \ref{theorem_weak convergence_Cremers_Kadelka}. To prove part (I) we consider the sequence of random vectors
	\begin{align*}
		(n^{-H}\zeta_{n}^{(q)}(s_{1}),...,n^{-H}\zeta_{n}^{(q)}(s_{p}))
	=	\sum_{j=-\infty}^{\infty} A_{nj} \varepsilon_{j}
	\end{align*}
for $s_{1},...,s_{p} \in \Sg$ with $A_{nj}=(n^{-(\frac{3}{2}-d(s_{a}))}a_{nj}(s_{a},t_{b}))_{a=1...p, b=1...q}$, $a_{nj}(s,t)$ like in (\ref{anj}) and $\varepsilon_{j}=\operatorname{diag}(\varepsilon_{j}(s_{1}),...,\varepsilon_{j}(s_{p}))$.
\begin{lemma} \label{Lemma_komplex_Multi_Bedingungen_anAnj}
If $h(s) \Intervall$, the sequence of operators $A_{nj}$ fulfills the conditions (\ref{Bedingungen_Operator}).
\end{lemma}
	\begin{proof}
	The operator norm may be calculated as $ \|A_{nj}\|_{op}
		=\max_{1\leq i \leq q} \sum_{l=1}^{p}| n^{d(s_{l})-\frac{3}{2}}a_{nj}(s_{l},t_{i})|$.
	The first condition holds, since
		\begin{align*}
		 	\sup_{j \in \Z} | n^{d(s)-\frac{3}{2}}a_{nj}(s,t)|	
		\leq 	n^{h(s)-\frac{3}{2}} \sup_{j \in \Z} \left\lbrace \sum_{k=1}^{\lfloor nt \rfloor }|v_{k-j}(s)|+\{nt\}|v_{\lfloor nt \rfloor+1-j}(s) | \right\rbrace 
		\sim \frac{t^{1-h(s)}}{1-h(s)}n^{-\frac{1}{2}}.
		\end{align*}
	The second one may be proved by using the Jensen inequality and the relationship between the variance of the piecewise linear process and $a_{nj}(s,t)$ given by
		\begin{align*}
		  E|\zeta_{n}(s,t)|^2
		=\sigma^2(s)\sum_{j=-\infty}^{\lfloor nt \rfloor +1} |a_{nj}(s,t)|^2,
		\end{align*}
	then
		\begin{align*}
		\sum_{j\in \Z}| n^{d(s)-\frac{3}{2}}a_{nj}(s,t)|^2
		=n^{2h(s)-3}\sum_{j\in \Z}|a_{nj}(s,t)|^2
		=n^{2h(s)-3}\frac{1}{\sigma^2(s)}E|\zeta_{n}(s,t)|^2
		\end{align*}
	and with the help of (\ref{equality_Convergence_behaviour_piecewise_linear_process}) and the inequalities (\ref{inequality_realpart_c(s)}) and (\ref{inequality_c(s)_in_dependence_of_h})
		\begin{align*}
		\lim_{n \to \infty} \frac{1}{\sigma^2(s)}n^{2h(s)-3}E|\zeta_{n}(s,t)|^2
		=&\frac{\frac{1}{2}(c(s)+\overline{c(s)})}{(3-2h(s))(1-h(s))}t^{3-2h(s)}	\\
		\leq&  \max\{t,t^2\} \left( \frac{1}{(1-h(s))^2}+\frac{1}{(1-h(s))(2h(s)-1)} \right).
		\end{align*}
	\end{proof}
We construct the process
	\begin{align*}
	(n^{-H}\tilde{\zeta}_{n}^{(q)}(s_{1}),...,n^{-H}\tilde{\zeta}_{n}^{(q)}(s_{p}))
	=\sum_{j=-\infty}^{\infty} A_{nj} \tilde{\varepsilon}_{j}.
	\end{align*}
Since $\{\tilde{\varepsilon}_{j}\}$ are zero-mean Gaussian stochastic processes with the same covariance operator as $\{\varepsilon_{j}\}$ and
	\begin{align*}
	\lim_{n\to \infty}
	n^{-(\frac{3}{2}-d(r_{i}))}n^{-(\frac{3}{2}-d(s_{i}))}
	E(\tilde{\zeta}_{n}(r_{i},t_{j})\tilde{\zeta}_{n}(s_{i},u_{j}))
	=E(G(r_{i},t_{j})G(s_{i},u_{j})),
	\end{align*}
concerning to (\ref{equality_Convergence_behaviour_piecewise_linear_process}), the process converges in distribution to $G$.
Calculating
	\begin{align*}
	E\|n^{-H}\zeta_{n}^{(q)}(s)\|^2_{\C^{q}}
	=\sum_{i=1}^{q}E[|n^{-(\frac{3}{2}-d(s))}\zeta_{n}(s,t_{i})|^2]
	\hspace*{0.2cm}
\text{and}
\hspace*{0.2cm}
	E\|G^{q}(s)\|^2
	=\sum_{i=1}^{q}E[|G(s,t_{i})|^2],
	\end{align*}
the second point in theorem \ref{theorem_weak convergence_Cremers_Kadelka} follows by setting $r=s$ and $t=u$ in (\ref{equality_Convergence_behaviour_piecewise_linear_process}). We prove point (III) in the next lemma.
 
\begin{lemma} \label{Lemma_integrierbare_Funktion_Invarianzpr_Multiplop}
There exists a $\mu$-integrable function $g$ such that $|E|n^{-(\frac{3}{2}-d(s))}\zeta_{n}(s,t)|^{2}| \leq g(s)$ for each $s$ in $\Sg$ and all $n \in \N$, if the integrals
	\begin{align*}
	\int_{\Sg} \frac{\sigma^2(s)}{(1-h(s))^2}\mu(ds)
	\hspace*{0.2cm} \text{ and } \hspace*{0.2cm}
	\int_{\Sg} \frac{\sigma^2(s)}{(1-h(s))(2h(s)-1)}\mu(ds)
	\end{align*}	
are finite.
\end{lemma}
	\begin{proof}
	Using the triangular inequality
		\begin{align*}
			|E|\zeta_{n}(s,t)|^{2}|
		\leq& 2 \sum_{k=1}^{\lfloor nt \rfloor} (\lfloor nt \rfloor -k) |\gamma_{k}(s)|+
			2\{nt\}\sum_{k=1}^{\lfloor nt \rfloor} |\gamma_{k} (s)|+
			(\lfloor nt \rfloor +\{nt\}^2)|\gamma_{0}(s)|.
		\end{align*}
Each particular summand may be restricted as follows
		\begin{align*}
			\sum_{k=1}^{\lfloor nt \rfloor}(\lfloor nt \rfloor-k) |\gamma_{k} (s)|
		\leq& \lfloor nt \rfloor^{3-2h(s)} \frac{1}{2} \sigma^2(s) \left( \frac{1}{(1-h(s))^2}+\frac{1}{(1-h(s))(2h(s)-1)} \right),\\
		\sum_{k=1}^{\lfloor nt \rfloor} |\gamma_{k} (s)|
		\leq &\lfloor nt \rfloor^{2-2h(s)} \sigma^2(s)\frac{1}{2}\left( \frac{1}{(1-h(s))^2}+\frac{1}{(1-h(s))(2h(s)-1)} \right) 
	\intertext{and}
			|\gamma_{0}(s)|
		\leq &	\sigma^2(s)\sum_{j=0}^{\infty}(j+1)^{-2h(s)}	\\
		\leq & \sigma^2(s)\left( 1 + \int_{1}^{\infty} x^{-2h(s)}dx\right)	
		=	\sigma^2(s)\left( 1 + \frac{1}{2h(s)-1}\right) .
		\end{align*}
	\end{proof}
Altogether $E\|n^{-H}\zeta_{n}^{q}(s)\|^2_{\C^{q}} \leq q*g(s)=:f(s)$,
where $f$ is integrable.

\subsubsection{Tightness}
To investigate tightness of the process $\{(n^{-H}\zeta_{n}(t))_{ n \in \N} | t \in [0,1] \}$ we use proposition \ref{proposition_tightness}.
The first condition follows since $(n^{-H}S_{n})_{n \in \N }$ converges in distribution in $L_{2}(\mu,\C)$ referring to theorem \ref{theorem_complex_multiplication_operator_central_limit}. Then $(n^{-H}\zeta_{n}(t))_{ n \in \N}$ converges in distribution in $L_{2}(\mu,\C)$ for each $t\in [0,1]$
\vspace*{0.2cm}	\\
The next lemma proves the second condition in proposition \ref{proposition_tightness}.
	\begin{lemma} \label{lemma_tightness_complex}
	If $h(s) \Intervall$ and the integrals
		\begin{equation*}
		E\left[ \int_{\mathbb{S}}\frac{|\varepsilon_{0}(v)|^{2}}{[1-h(v)]^2}\mu(dv)\right]^{\frac{p}{2}}
		\text{ , for $p\geq 2$ and } \hspace*{0.2cm}
		\int_{\mathbb{S}}\frac{\sigma^{2}(v)}{[1-h(v)][2h(v)-1]}\mu(dv)
		\end{equation*}
	are finite, then there exists a constant $C$, such that
		\begin{align*}
		E\| n^{-H}[\zeta_{n}(t)-\zeta_{n}(u)]\|^p \leq C|t-u|^{\frac{(3-2\bar{h})p}{2}}
		\hspace*{0.2cm} \text{ , for } \hspace*{0.2cm}
		n\geq 1
		\end{align*}
	with $\bar{h}=\esssup_{s \in \Sg}h(s)$.
	\end{lemma}
The following lemma is part of the proof of the previous one.
\begin{lemma} \label{Hilfslemma_komplex_Multi_Straffheit}
Let $0\leq u<t\leq 1$ and $\{nt\}=\{nu\}=0$. If $h(s) \Intervall$, then
	\begin{align*}
			n^{-(3-2h(s))}\sum_{j=-\infty}^{nt} \left| \sum_{k=nu+1}^{nt} v_{k-j}(s) \right|^2
	\leq	\left[ \frac{2}{(1-h(s))^2} +\frac{1}{2h(s)-1}\right] |t-u|^{3-2h(s)}.
	\end{align*}
\end{lemma}
We refer to \citet{Chara_Rack_Operator}) for the proofs. They may be easily extended using the triangular inequality and the inequalities (\ref{inequality_realpart_c(s)}) and (\ref{inequality_c(s)_in_dependence_of_h}) as we did so far.
\end{proof}

  \bibliographystyle{agsm}
  \bibliography{bibpaper}
  
\end{document}